\documentclass{article}
\usepackage[left=3cm,right=3cm,top=3cm,bottom=3cm]{geometry} % page settings
\usepackage{amsmath,amssymb} % provides many mathematical environments & tools
\usepackage{amsthm}
\usepackage{tikz}
\usepackage{color}
\usepackage{hyperref}
\usetikzlibrary{patterns}
\usetikzlibrary{intersections}
\usetikzlibrary{arrows,positioning}
\usepackage{graphicx}
\usepackage{pgfplots}
\pgfplotsset{compat=1.14}
\usetikzlibrary{plotmarks}
  \usepgfplotslibrary{patchplots}
  \usepackage{grffile}
\usepgfplotslibrary{fillbetween}
\usepackage{float}
\usepackage{fancyhdr}
\usepackage{caption}
\usepackage{subcaption}

\newcommand*{\rom}[1]{\expandafter\@slowromancap\romannumeral #1@}
\newcommand{\gettikzxy}[3]{%
  \tikz@scan@one@point\pgfutil@firstofone#1\relax
  \edef#2{\the\pgf@x}%
  \edef#3{\the\pgf@y}%
}

\newcommand{\sA}{\mathcal A}
\newcommand{\sB}{\mathcal B}
\newcommand{\sC}{\mathcal C}
\newcommand{\sD}{\mathcal D}
\newcommand{\sF}{\mathcal F}

\newcommand{\sI}{\mathcal I}
\newcommand{\sL}{\mathcal L}
\newcommand{\sM}{\mathcal M}
\newcommand{\sN}{\mathcal N}
\newcommand{\sP}{\mathcal P}
\newcommand{\sQ}{\mathcal Q}
\newcommand{\sR}{\mathcal R}
\newcommand{\sS}{\mathcal S}
\newcommand{\sT}{\mathcal T}
\newcommand{\sV}{\mathcal V}

\newcommand{\R}{\mathbb R}
\newcommand{\E}{\mathbb E}
\newcommand{\F}{\mathbb F}
\newcommand{\Q}{\mathbb Q}
\newcommand{\bM}{\mathbb M}
\newcommand{\Prob}{\mathbb P}

\newcommand{\arginf}{\mbox{arginf}}
\newcommand{\argsup}{\mbox{argsup}}
\newcommand{\Leb}{\mbox{Leb}}
\newtheorem{thm}{Theorem}
\newtheorem{prop}{Proposition}

\newtheorem{lem}{Lemma}

\newtheorem{rem}{Remark}

\newtheorem{defn}{Definition}

\begin{document}
\title{The left-curtain martingale coupling in the presence of atoms}
\author{David Hobson\thanks{E-mail: \textit{d.hobson@warwick.ac.uk}}\hspace{3mm}  and  \hspace{3mm}Dominykas Norgilas\thanks{E-mail: \textit{d.norgilas@warwick.ac.uk}}\\
Department of Statistics, University of Warwick\\Coventry CV4 7AL, UK}
\date{\today}
\maketitle

\begin{abstract}
Beiglb\"{o}ck and Juillet~\cite{BeiglbockJuillet:16} introduced the left-curtain martingale coupling of probability measures $\mu$ and $\nu$, and proved that, when the initial law $\mu$ is continuous, it is supported by the graphs of two functions. We extend the later result by constructing the generalised left-curtain martingale coupling and show that for an arbitrary starting law $\mu$ it is characterised by two appropriately defined lower and upper functions.

As an application of this result we derive the model-independent upper bound of an American put option. This extends recent results of Hobson and Norgilas~\cite{HobsonNorgilas:17} on the atom-free case.
\end{abstract}

\section{Introduction}

Given two probability measures $\mu$ and $\nu$ on $\R$ and a payoff function $c:\R\times\R\to\R$, the classical problem of optimal transport is to construct a joint distribution for random variables $Z_1\sim\mu$ and $Z_2\sim\nu$ which minimises $\mathbb{E}[c(Z_1,Z_2)]$. Beiglb\"{o}ck et al.~\cite{BeiglbockHenrylaborderePenkner:13} and Galichon et al.~\cite{GalichonLabordereTouzi:14} introduced a martingale version of the transportation problem and related it to the problem of finding model-independent bounds of exotic derivatives in mathematical finance. Given $\mu$ and $\nu$ in convex order, the basic problem of martingale optimal transport (MOT) is to construct a martingale $M$, with $M_1\sim\mu,M_2\sim\nu$, which minimises $\mathbb{E}[c(M_1,M_2)]$. In this setting a martingale transport or coupling can be identified with a measure $\pi$ on $\R^2$ with univariate marginals $\mu$ and $\nu$, and such that $\int_{y\in\R} \int_{x \in A} (y-x) \pi(dx,dy) =0$ for all Borel sets $A$, and the MOT is to find $\pi$ to minimise $\int \int c(x,y) \pi(dx,dy)$. In the context of mathematical finance this problem was first studied in Hobson and Neuberger~\cite{HobsonNeuberger:12} for the payoff $c(x,y)=-\lvert y-x\lvert$. 
	
	Beiglb\"{o}ck and Juillet~\cite{BeiglbockJuillet:16} introduced the notion of a left-monotone martingale coupling and established that for (arbitrary) fixed marginals $\mu$ and $\nu$ in convex order there exists a unique such coupling (called the left-curtain martingale coupling and denoted by $\pi_{lc}$). The left-curtain martingale coupling may be viewed as a martingale analogue to the monotone Hoeffding-Fr\`{e}chet coupling in classical optimal transport. The authors also proved the optimality of $\pi_{lc}$ for a specific class of payoff functions. Henry-Labord\`{e}re and Touzi~\cite{Touzi:16} extended the results of Beiglb\"{o}ck and Juillet~\cite{BeiglbockJuillet:16} and showed optimality for a wider class of payoff functions. Beiglb\"{o}ck et al.~\cite{BeiglbockHenryLabordereTouzi:17} analysed the left-curtain coupling further and gave a simplified proof of uniqueness under the additional assumption that $\mu$ is continuous. Juillet \cite{Juillet:16} proved that $\pi_{lc}$ is continuous, and thus, for general distributions, it can be approximated by the left-curtain couplings corresponding to `nice' (e.g. finitely supported or continuous) initial and/or target laws. A number of further articles investigate the properties and extensions of $\pi_{lc}$, see Beiglb\"{o}ck et al. \cite{beiglbockGeom:17, BeiglbockHenryLabordereTouzi:17}, Nutz et al. \cite{NutzStebegg:18, NutzStebeggTan:17}.

Beiglb\"{o}ck and Juillet~\cite{BeiglbockJuillet:16} also established a martingale version of the fundamental Brenier's \cite{Brenier:87} result in the classical optimal tansport which states that, for a sufficiently regular initial distribution $\mu$, the optimal transport map is unique and supported by the graph of the gradient of some convex function (a monotonically increasing function in one dimension). In particular, the authors showed that under the assumption that the initial law $\mu$ is continuous, the left-curtain martingale coupling  is supported by the graphs of lower and upper functions $T_d$ and $T_u$, respectively, so that $M_2 \in \{ T_d(M_1),T_u(M_1) \}$. Henry-Labord\`{e}re and Touzi~\cite{Touzi:16} gave an explicit construction of $T_d$ and $T_u$ using differential equations. However, when $\mu$ has an atom at $x$ the element $\pi^x_{lc}(\cdot)$ in the disintegration $\pi_{lc}(dx,dy) = \mu(dx) \pi_{lc}^x(dy)$ becomes a measure with support on non-trivial subsets of $\R$ and not just on a two point set. Then we cannot construct functions $(T_d,T_u)$, unless we allow them to be multi-valued. 

Our goal in this paper is to show how by changing our viewpoint we can again recover the
property that $M_2$ takes values in a two-point set. The idea is to write $M_1 = h(Z)$ for 
a continuous random variable $Z$ (in fact we take $Z \equiv U \sim U(0,1)$) and then to find $f_{Z,h}$ and $g_{Z,h}$ such that $M_2 \in \{f_{Z,h}(Z), g_{Z,h}(Z) \}$. Then, although there is uniqueness at the level of martingale couplings $\pi$, when $\mu$ contains atoms there are many possible choices of $(f_{Z,h},g_{Z,h})$, even for fixed $Z$ and monotonic increasing $h$.   
Nonetheless, we show that amongst this set there is an essentially unique choice $(f_{Z,h},g_{Z,h})$ with a special monotonicity property. 

The motivation for this extension of the left-curtain martingale coupling comes from mathematical finance. The recent study of American put options in Hobson and Norgilas~\cite{HobsonNorgilas:17} highlights the role of the left-curtain martingale coupling in finding the model-independent upper bound on the price of the American put. When $\mu$ is continuous the authors show how the optimal martingale coupling and the optimal stopping time can be obtained from the functions $f=T_d$ and $g=T_u$ which arise in the construction of the left-curtain coupling. In particular, for the optimal model there is a Borel subset of $\R$, say $B$, such that it is optimal to stop at time-1 if $M_1\in B$, and at time-2 otherwise. Moreover, the structure of $f$ and $g$ allows us to identify the cheapest superhedging strategy that supports the price of the American put.

If $\mu$ has atoms then the situation becomes more delicate, essentially because we must allow for a wider range of possible candidates for exercise determining sets $B$. On atoms of $\mu$ we may want to sometimes stop and sometimes continue, although we must still take stopping decisions which do not violate the martingale property. As the stopping decision in the continuous case is based on the natural filtration of the martingale $M$, if $M_1$ ends up at the atom of $\mu$, then it is not clear, using only the structure of $f$ and $g$, what part of mass at time-1 should be stopped and what part should be allowed to continue. This is the reason why we must extend the notion of the left-curtain martingale coupling.

The main effort in this article is in proving Theorem~\ref{thm:atoms} which extends the left-curtain martingale coupling to the presence of atoms in the starting law $\mu$. We show that this extended coupling is again characterised by lower and upper functions, $R$ and $S$, respectively.  However, while $f$ and $g$ are multi-valued on the atoms of $\mu$, $R$ and $S$ remain well-defined. Then our second achievement is to show how the structure of $R$ and $S$ can be used to characterise the model and stopping rule which achieves the highest possible price for the American put, and the cheapest superhedge. This generalises results of Hobson and Norgilas~\cite{HobsonNorgilas:17}: for arbitrary $\mu$ and $\nu$, the highest model based price of the American put is equal to the cost of the cheapest superhedge.

\section{Preliminaries and set-up}
\label{sec:setup}

Let $\sM(\R^n)$ (respectively $\sP(\R^n)$) be the set of measures (respectively probability measures) on $\R^n$. Given an integrable $\eta\in\sM(\R)$, i.e. $\int_\R\lvert x\lvert\eta(dx)<\infty$, define $\bar{\eta} = \frac{\int_\R x \eta(dx)}{\int_\R \eta(dx)}$ to be a barycentre of $\eta$. Let $\sI_\eta$ with endpoints $\{ \ell_\eta, r_\eta \}$ be the smallest interval containing the support of $\eta$ (with the convention that finite endpoints are included).
Define $P_\eta : \R \mapsto \R^+$ by $P_\eta(k) = \int_{-\infty}^k (k-x) \eta(dx)$.
Then $P_\eta$ is convex and increasing, $\lim_{ z \downarrow -\infty} P_\eta(z)=0$, $\lim_{z \uparrow \infty} P_\eta(z) - \eta(\R)(z- \bar{\eta})^+=0$ and $\{k : P_{\eta}(k) > \eta(\R)(k - \bar{\eta})^+ \} \subseteq \sI_\eta$.
Note that $P_{\eta}$ is related to the potential $U_\eta$ defined by $U_\eta(k) : = - \int_{\R} |k-x| \eta(dx)$ by
$P_\eta(k) = \frac{1}{2}(- U_\eta(k) + (k- \bar{\eta}) \eta(\R))$. For $\eta \in \sP(\R)$ let $F_\eta$ be the distribution function of $\eta$ and let $G_\eta:(0,1) \mapsto \R$ be the quantile function of $\eta$, which is taken to be left-continuous unless otherwise stated.

Two measures $\eta$ and $\chi$ are in convex order, and we write $\eta \leq_{cx} \chi$, if and only if $\eta(\R)= \chi(\R)$, $\bar{\eta}= \bar{\chi}$ and $P_{\eta}(k) \leq P_{\chi}(k)$ on $\R$. Necessarily we must have
$\ell_\chi \leq \ell_\eta \leq r_\eta \leq r_\chi$. %If $\sI_D \subset \sI_\nu$ then we must have $\mu=\nu$ on $[\ell_r, \ell_D) \cup (r_D, r_\nu)$.
For any two probability measures $\eta,\chi\in\sP(\R)$ we write $\pi\in{\Pi}(\eta,\chi)$ if $\pi\in\sP(\R^2)$ and has first marginal $\eta$ and second marginal $\chi$. If $\pi\in{\Pi}(\eta,\chi)$ is such that the following martingale condition holds
\begin{equation}
\int_{x \in B} \int_{y \in \R} y \pi(dx,dy) = \int_{x \in B} \int_{y \in \R} x \pi(dx,dy) = \int_B x \eta(dx) \hspace{10mm}
\mbox{$\forall$ Borel $B \subseteq \R$},
\label{eq:martingalepi}
\end{equation}
we write $\pi\in{\Pi}_M(\eta,\chi)\subset{\Pi}(\eta,\chi)$ and say that $\pi$ is a martingale coupling of $\eta$ and $\chi$. By a classical result of Strassen \cite{Strassen:65}, ${\Pi}_M(\eta,\chi)$ is non-empty if and only if $\eta\leq_{cx}\chi$.
\begin{defn}[Hobson and Neuberger~\cite{HobsonNeuberger:17}]
Suppose $\mu \leq_{cx} \nu$.

Let $\sS = (\Omega, \sF, \Prob, \F = \{ \sF_0, \sF_1, \sF_2 \})$ be a filtered probability space. We say $M=(M_0,M_1,M_2)=(\bar{\mu},X,Y)$ is a $(\sS,\mu,\nu)$ consistent stochastic process and we write $M \in \bM(\sS, \mu, \nu)$ if
\begin{enumerate}
\item $M$ is a $\sS$-martingale
\item $\sL(M_1) = \mu$ and $\sL(M_2) = \nu$
\end{enumerate}
We say $(\sS,M)$ is a $(\mu,\nu)$-consistent model if $\sS$ is a filtered probability space and $M$ is a $(\sS,\mu,\nu)$ consistent stochastic process.
\end{defn}

Let $(\eta_n)_{n\geq1}$ be a sequence of probability measures in $\sP(\R)$. For $\eta\in\sP(\R)$, we write $\eta_n\xrightarrow{w}\eta$, and say $\eta_n$ converges weakly to $\eta$, if $\lim_{n\to\infty}\int fd\eta_n=\int fd\eta$ for all bounded and continuous functions $f$ on $\R$ (see Billingsley \cite{Billingsley:2013}).
If $\eta_n\xrightarrow{w}\eta$,
%and $U_{\eta_n}(z) \rightarrow U_\eta(z)$ then we have $\eta_n$ tends to $\eta$ in convex order, $\eta_n \xrightarrow{cx} \eta$. If also
if $\eta_n \leq_{cx} \eta$ and if $(\eta_n)_{n\geq1}$ is increasing in convex order, i.e. $\eta_n\leq_{cx}\eta_{n+1}$ for each $n$, then we write $\eta_n \uparrow_{cx} \eta$.
\begin{lem}\label{lem:approx}
Suppose $\mu\in\sP(\R)$ is integrable. Then there exists a sequence $(\mu_n)_{n\geq1}$ of finitely supported integrable measures in $\sP(\R)$  such that $\mu_n \uparrow_{cx} \mu$.
\end{lem}
\begin{proof}
Recall that for any $\eta\in\sP(\R)$, $U_\eta$ is concave, linear on each interval $I\subset\R$ with $\eta(I)=0$, $U_\eta(x)\leq-\lvert\bar{\eta}-x\lvert=U_{\delta_{\bar{\eta}}}(x)$ on $\R$, and $\lim_{\lvert x\lvert\to\infty}U_\eta(x)+\lvert\bar{\eta}-x\lvert=0$. Moreover, $\mu_n \uparrow_{cx} \mu$ if and only if $U_{\mu_n}\downarrow U_\mu$ pointwise, see Chacon~\cite{Chacon:77}. Let $\mathcal{U}_\mu$ be a set of piecewise linear concave functions $\tilde U:\R\to\R_-$ such that ${U}_\mu(x)\leq \tilde{U}(x)\leq U_{\delta_{\bar{\mu}}}(x)$. Then each $\tilde{U}\in\mathcal{U}_\mu$ corresponds to a finitely supported integrable probability measure $\tilde{\mu}$ on $\R$ such that $\delta_{\bar{\mu}}\leq_{cx}\tilde{\mu}\leq_{cx}\mu$. Finally, Chacon and Walsh \cite{Chacon:76} provide a sequence of functions $(\tilde U_n)_{n\geq1}$ in $\mathcal{U}_\mu$, such that $\tilde U_n\downarrow U_\mu$ pointwise, proving our claim.
\end{proof}

\section[General case]{An extension of the left-curtain mapping to the general case}
\label{sec:atoms}
In this section we construct a new representation of the left-curtain martingale coupling of Beiglb\"{o}ck and Juillet~\cite{BeiglbockJuillet:16}. Our approach is to construct $(X,Y)$ from a pair of independent uniform $U(0,1)$ random variables $U$ and $V$. The construction of $X$ is straightforward: we set $X= G_{\mu}(U)$. % where $G_\mu$ is the left-continuous quantile function of a random variable with law $\mu$.

It remains to construct $Y$. First we consider the case of a point mass at $w$, $\mu = \delta_w$, and show how to construct functions $R=R_{\mu,\nu}$ and $S=S_{\mu,\nu}$ with $R_{\mu,\nu}(u) \leq G_\mu(u) \leq S_{\mu,\nu}(u)$, such that if $X=G_{\delta_w}(U)=w$ and $Y \in \{ R(U),S(U) \}$ with $\Prob(Y = R(u)|U=u) = \frac{S(u)-G(u)}{S(u)-R(u)}$ then $Y$ has law $\nu$. In particular, conditional on $U=u$, $Y$ takes values in $\{ R(u), S(u) \}$ and satisfies $\E[Y|U=u] = G_\mu(u)$. Second, we show how this result extends to the case of a measure $\mu$ consisting of finitely many atoms. Third, for the case of general $\mu$ we construct an approximation $(\mu_n)_{n \geq 1}$ of $\mu$ and associated functions $(R_n, G_n, S_n)_{n \geq 1}$ where each $\mu_n$ is finitely supported. We show that we can define limits $(R,G,S)$ such that $(R,G,S)$ can be used to construct a martingale $M = (M_0 = \bar{\mu}, M_1=X, M_2= Y)$ with the property that $X=G(U)$ and $Y \in \{ R(U),S(U) \}$ and such that $\sL(X)=\mu$ and $\sL(Y) = \nu$. The functions $R,S:(0,1) \mapsto \R$ we define have the properties
\begin{equation}R(u) \leq G(u) \leq S(u); \hspace{10mm} \mbox{$S$ is increasing;} \hspace{10mm} \mbox{for $0 < u < v < 1$, $R(v) \notin (R(u),S(u))$}.
\label{eq:RSproperties}
\end{equation}

We suppose $\mu\leq_{cx}\nu$ are fixed and given and we abbreviate the quantile function $G_\mu$ by $G$. The aim of this section is to prove the following theorem:
\begin{thm}
\label{thm:atoms}
There exist functions $R,S:(0,1) \mapsto \R$ satisfying \eqref{eq:RSproperties}
such that if we define $X(u,v)=X(u)=G(u)$ and $Y(u,v) \in \{R(u),S(u) \}$ by $Y(u,v) = G(u)$ on $G(u)=S(u)$ and
\begin{equation} Y(u,v) = R(u) I_{\{ v \leq \frac{S(u) - G(u)}{S(u)-R(u)} \}} +  S(u) I_{ \{ v > \frac{S(u) - G(u)}{S(u)-R(u)} \} }
\label{eq:YUVdef}
\end{equation}
otherwise, and if $U$ and $V$ are independent $U(0,1)$ random variables then
$M = (\bar{\mu},X(U),Y(U,V))$ is a $\F = (\sF_0 = \{ \emptyset, \Omega \}, \sF_1 = \sigma(U), \sF_2 = \sigma(U,V) \})$-martingale for which $\sL(X) = \mu$ and $\sL(Y) = \nu$.

In particular, if $\Omega = (0,1) \times (0,1)$, $\sF = \sB(\Omega)$, $\Prob = \Leb(\Omega)$, if $\F$ and $M$ are defined as above and if $\sS = (\Omega, \sF, \F, \Prob)$ then $(\sS,M)$ is a $(\mu,\nu)$-consistent model.
\end{thm}

%{\color{red} \begin{rem}[Ignore]
%Usually $G_\mu(u)=\inf\{x:u\leq F_\mu(x)\}$ is defined only on $(0,1)$. If we would define it on $[0,1]$, we would have, for example, $G_\mu(0)=-\infty$. On the other hand, $G_\mu(1-)=\lim_{u\to1}G(u)=r_\mu$, and $G(0-)=\lim_{u\to0}=\ell_\mu$. Then, even if $\sI_\mu=\R$, $G$ does not take values in $[-\infty,+\infty]$.
%\end{rem}}
\begin{rem}
For $n\geq1$, let $\pi_{lc}^n$ be the left-curtain coupling of the initial law $\mu_n$ (consisting of $n$ atoms) and target law $\nu$. Juillet \cite{Juillet:16} proved that if $(\mu_n)_{n\geq1}$ converges weakly to $\mu$ then $(\pi^n_{lc})_{n\geq1}$ converges weakly to the left-curtain coupling of $\mu$ and $\nu$.

Here we argue differently. We use the fact that $\pi_{lc}^n$ can be represented by an explicitly constructed triple $(S_n,G_n,R_n)$. Then, by sending $n\to+\infty$, we show that the limiting functions give rise to the left-monotone martingale coupling, and thus also to $\pi_{lc}$, of $\mu$ and $\nu$.
\end{rem}
When $\mu$ is continuous and $f$ and $g$ are well-defined the construction of this section is related to that of Beiglb\"{o}ck and Juillet~\cite{BeiglbockJuillet:16} (see also Henry-Labord\`{e}re and Touzi~\cite{Touzi:16}) via the relationships $S =g \circ G_\mu$ and $R=f\circ G_\mu$. Suppose $\nu$ is also continuous and fix $x$. Then under the left-curtain martingale coupling $\{f(x),g(x)\}$ with $f(x)\leq x\leq g(x)$ are solutions to the mass and mean conditions
 \begin{align}
 \int^x_f\mu(dz)&=\int^g_f\nu(dz),\label{eq:ContMass}\\
 \int^x_fz\mu(dz)&=\int^g_fz\nu(dz).\label{eq:ContMean}
 \end{align}
 When $\mu$ has atoms, $G_{\mu}$ has intervals of constancy and $f$ and $g$ are multi-valued, but $R$ and $S$ remain well-defined. See Figure~\ref{fig:RGSfg}. Then, for general $\mu$ and $\nu$, the appropriate generalisations of \eqref{eq:ContMass} and \eqref{eq:ContMean} are
 \begin{align}
 \int_{(R(u),G(u))}\mu(dz)+\overline\lambda^\mu_u&=\int_{(R(u),S(u))}\nu(dz)+\underline\lambda^\nu_u+\overline\lambda^\nu_u,\label{eq:mass}\\
 \int_{(R(u),G(u))}z\mu(dz)+\overline\lambda^\mu_uG(u)&=\int_{(R(u),S(u))}z\nu(dz)+\underline\lambda^\nu_uR(u)+\overline\lambda^\nu_uS(u),\label{eq:mean}
 \end{align}
 respectively, where the quantities $0\leq\overline\lambda^\mu_u\leq\mu(\{G(u)\})$, $0\leq\underline\lambda^\nu_u\leq(\nu-\mu)(\{R(u)\})$, $0\leq\overline\lambda^\nu_u\leq\nu(\{S(u)\})$ are uniquely determined by the triple $(R,G,S)$. Essentialy, \eqref{eq:mass} is preservation of mass condition and \eqref{eq:mean} is preservation of mean condition. Together they give the martingale property.
 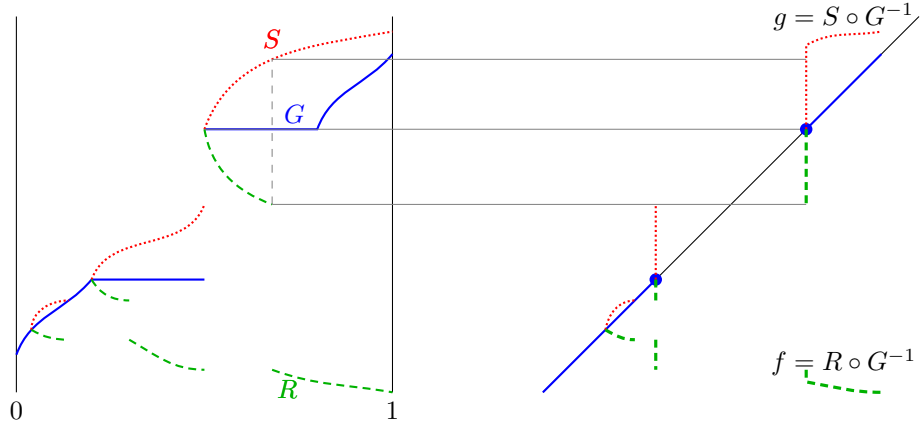
\begin{figure}[H]

\centering
\begin{tikzpicture}[scale=1,
declare function={	
k1=2.1;
k2=1;
a(\x)=((k1-\x)*(\x<k1))-k1-1;
b(\x)=((k2-\x)*(\x<k2))-k1-1;
x1=-6.8;
x2=-6.3;
z1=-5.2;
z2=-4.7;
 }]

               \draw[ black] (-7,0)--(-7,5);
               \draw[black] (-2,0)--(-2,5);

               \draw[name path=diag, black] (0,0) -- (5,5);
                %payoffs

    %G
    \draw[blue,thick, name path=g1] (-7,0.5) to[out=70, in=230] (-6,1.5) -- (-4.5,1.5) ;
     \draw[blue,thick, name path=g2] (-4.5,3.5) -- (-3,3.5) to[out=70, in=230] (-2,4.5)  ;

     %S
      \path [name path=lineA](x1,0) -- (x1,5);
	\path [gray, very thin, name intersections={of=lineA and g1}] (x1,0) -- (intersection-1);
	\coordinate (temp1) at (intersection-1);

	 \path [name path=lineA](x2,0) -- (x2,5);
	\path [gray, very thin, name intersections={of=lineA and g1}] (x2,0) -- (intersection-1);
	\coordinate (temp2) at (intersection-1);
	
	\draw[red,thick, densely dotted,name path=s1] (temp1) to[out=90, in=180] (temp2) ;
	\draw[red,thick, densely dotted,name path=s2] (-6,1.5) to[out=70, in=250] (-4.5,2.5) ;
	\draw[red,thick, densely dotted,name path=s3] (-4.5,3.5) to[out=70, in=190] (-2,4.8) ;
	
	%R
	\gettikzxy{(temp1)}{\l}{\k};
	\gettikzxy{(temp2)}{\n}{\m};
	\draw[black!30!green,thick, densely dashed, name path=r1] (temp1) to[out=330, in=180] (\n,0.7) ;
	\draw[black!30!green,thick, densely dashed, name path=r2] (-6,1.5) to[out=300, in=180] (-5.5,\m) ;
	\draw[black!30!green,thick, densely dashed, name path=r3] (-5.5,0.7) to[out=330, in=180] (-4.5,0.3) ;
	\draw[black!30!green,thick, densely dashed, name path=r4] (-4.5,3.5) to[out=280, in=160] (-3.6,2.5) ;
	\draw[black!30!green,thick, densely dashed, name path=r5] (-3.6,0.3) to[out=340, in=170] (-2,0) ;
	%\draw[red,thick, name path=s1] (-6,1.5) to[out=70, in=250] (-4.5,2.5) ;
	%\draw[red,thick, name path=s1] (-4.5,3.5) to[out=45, in=210] (-2,4.8) ;
	
     %\draw[red,thick] (-6.7,0.5) to[out=70, in=200] (-6.2,1.5) -- (-4.5,1.5) ;
     %\draw[blue,thick] (-7,0.5) to[out=70, in=230] (-6,1.5) -- (-4.5,1.5) ;
     %\draw[blue,thick] (-7,0.5) to[out=70, in=230] (-6,1.5) -- (-4.5,1.5) ;

    %line
    \path[gray, very thin] (\n,0.7) -- (temp2);
    \path[gray, very thin] (\n,0.7) -- (0.7,0.7);
    \path[gray, very thin]  (temp2)--(\m,\m);
    \path[gray, very thin]  (-2,4.5)--(4.5,4.5);

      \path [name path=lineA](-5.5,0.7) -- (-5.5,5);
	\path [gray, very thin, name intersections={of=lineA and s2}] (-5.5,0.7) -- (intersection-1);
	\coordinate (vs2) at (intersection-1);
	\gettikzxy{(vs2)}{\aa}{\bb}
	 \path[gray, very thin]  (vs2)--(\bb,\bb);

	\path[gray, very thin] (-4.5,0.3) -- (-4.5,3.5);
	\draw[gray, very thin] (-4.5,3.5) -- (3.5,3.5);
	\path[gray, very thin] (-4.5,0.3) -- (0.3,0.3);
	
	\path[gray, very thin] (-3.6,0.3) -- (-3.6,2.5) ;
	\draw[gray, very thin] (-3.6,2.5) -- (3.5,2.5) ;
	
	 \path [name path=lineA](-3.6,2.5) -- (-3.6,5);
	\draw [gray, very thin, dashed, name intersections={of=lineA and s3}] (-3.6,2.5) -- (intersection-1);
	\coordinate (test1) at (intersection-1);
	\gettikzxy{(test1)}{\testx}{\testy}
	 \draw[gray, very thin]  (test1)--(3.5,\testy);

      \path [name path=lineA](-3,3.5) -- (-3,5);
	\path [gray, very thin, name intersections={of=lineA and s3}] (-3,3.5) -- (intersection-1);
	\coordinate (vs3) at (intersection-1);
	\gettikzxy{(vs3)}{\aaa}{\bbb}
	 \path[gray, very thin]  (vs3)--(\bbb,\bbb);

	\path [name path=lineA](-3,3.5) -- (-3,0);
	\path [gray, very thin, name intersections={of=lineA and r5}] (-3,3.5) -- (intersection-1);
	\coordinate (vr5) at (intersection-1);
	\gettikzxy{(vr5)}{\aaaa}{\bbbb}
	 \path[gray, very thin]  (vr5)--(\bbbb,\bbbb);
	
	  \path[gray, very thin]  (vs2)--(\bb,\bb);
	   \path[gray, very thin]  (vs2)--(\bb,\bb);

	% atoms on diagonal
	\node [scale=0.5, shape=circle, fill, blue] at (3.5,3.5) {} ;
	\node [scale=0.5, shape=circle, fill, blue] at (1.5,1.5) {} ;

      %T_u
     \draw[blue,thick] (0,0) -- (1.5,1.5);
      \draw[blue,thick] (3.5,3.5) -- (4.5,4.5);
    	\draw[red,densely dotted, thick] (\k,\k) to[out=80, in=180] (\m,\m);
     \draw[red,densely dotted, thick] (1.5,1.5) -- (1.5,2.5);
     %\draw[red,dashed,thick] (1.5,2.5) -- (2.5,2.5)--(3.5,3.5);
     \draw[red,densely dotted, thick] (3.5,3.5) -- (3.5,\bbb) to[out=30, in=190] (4.5,4.8);

     %T_d
     \draw[black!30!green, densely dashed,very thick] (\k,\k) to[out=330, in=180] (\m,0.7);
     \draw[black!30!green,densely dashed,very thick] (1.5,1.5) -- (1.5,\m);
     \draw[black!30!green,densely dashed,very thick] (1.5,0.7) -- (1.5,0.3);
     \draw[black!30!green,densely dashed,very thick] (3.5,3.5) -- (3.5,2.5);
     \draw[black!30!green,densely dashed,very thick] (3.5,0.3) -- (3.5,\bbbb) to[out=350, in=180] (4.5,0);

\node[red] at (-3.6,4.7) {$S$};
\node[blue] at (-3.3,3.7) {$G$};
\node[black!30!green] at (-3.4,0.04) {$R$};

\node[red] at (-3.6,4.7) {$S$};
\node[black!30!green] at (-3.4,0.04) {$R$};

\node[below] at (-7,0) {$0$};
\node[below] at (-2,0) {$1$};
\node[black] at (4,0.4) {$f=R \circ G^{-1}$};
\node[black] at (4,5) {$g=S \circ G^{-1}$};

\end{tikzpicture}

\caption{Sketch of $R,G,S$ and the corresponding $f$ and $g$. On the atoms of $\mu$, $G$ is flat, and $f$ and $g$ are multi-valued, but $R$ and $S$ remain well-defined.}
\label{fig:RGSfg}
\end{figure}
%\begin{rem}
%Given a triple $(R,G,S)$ each of which is a function from $[0,1]$ to $\R$, and jointly satisfying $R(u) \leq G(u) \leq S(u)$ we can construct a coupling of a pair of probability measures. Specifically for $(U,V)$ a pair of independent uniform random variables let $X=X(U)$ and $Y=Y(U,V)$ where $X(u) = G(u)$ and $Y(u,v)$ is given by \eqref{eq:YUVdef} on $(S(u)-G(u))(G(u)-R(u))>0$ and $Y(u,v)=G(u)$ otherwise. (Note, we need not assume the monotonicity properties in \eqref{eq:RSproperties}, just the ordering $R \leq G \leq S$.) Let $\mu_{\sV} = \sL(X)$ and $\nu_{\sV}  = \sL(Y)$ under this construction. Then, subject to integrability conditions, we have a martingale coupling of $\mu_{\sV}$ and $\nu_{\sV}$. In this way we can consider the construction as a map $\sV$ from triples $(R,G,S)$ with $R(u) \leq G(u) \leq S(u)$ to pairs of probability measures in convex order.
%\label{rem:sV}
%\end{rem}

\subsection{The case where $\mu$ is a point mass}
\label{ssec:pointmass}
The goal in this section is to prove Theorem~\ref{thm:atoms} in the special case where $\mu$ is a point mass. We assume that $\mu$ is a unit atom at $w$ and $\nu$ is centred at $w$. Then $\mu = \delta_w \leq_{cx} \nu$.

Let $P(k)=P_\nu(k) = \int_{-\infty}^\infty (k-z)^+ \nu(dz)$. Then $P(k) \geq (k-w)^+$. For $p \in [0,P(w)]$ define $\alpha:[0,P(w)] \mapsto [w,\infty]$ and $\beta:[0,P(w)] \mapsto [-\infty,w]$ by
\begin{equation}
\alpha(p) = \arginf_{k>w} \left\{ \frac{P(k) - p}{k - w} \right\}; \hspace{20mm} \beta(p) = \argsup_{k<w} \left\{ \frac{p - P(k)}{w-k} \right\},
\label{eq:alphabetadef}
\end{equation}
see Figure~\ref{fig:alphabetaab}.
Then $\alpha$ is decreasing and $\beta$ is increasing. Since the $\arginf$ and $\argsup$ may not be uniquely defined (this happens when $\nu$ has intervals with no mass) we avoid indeterminacy by assuming that $\alpha$ and $\beta$ are right-continuous. (We also set $\alpha(P(w))= \inf\{ z>w: F_\nu(z) > F_\nu(w) \}$ and
$\beta(P(w))= \sup\{ z<w: F_\nu(z) < F_\nu(w-) \}$. Note that $\alpha(0)= r_\nu$ and $\beta(0)= \ell_\nu$.) If $\nu$ has atoms then $\alpha$ and $\beta$ may fail to be strictly monotonic.

For $p \in (0,P(w))$ define also
\begin{equation}
a(p) = \inf_{k>w} \frac{P(k) - p}{k-w}  =\frac{P(\alpha(p)) - p}{\alpha(p)-w}; \hspace{20mm} b(p) = \sup_{k<w} \frac{p - P(k)}{w-k} = \frac{p - P(\beta(p))}{w-\beta(p)}. \label{eq:abdef}
\end{equation}
Extend the representations to $[0,P(w)]$ by taking limits.
%where the second representations in each case are valid only if $\alpha(p)>w$ or $\beta(p)<w$ respectively.
Then $a:[0,P(w)] \mapsto [P'(w+),1]$ is decreasing and $b:[0,P(w)] \mapsto [0,P'(w-)]$ is increasing. We have the representations
\[ a(p) = 1 - \int_0^p \frac{dq}{\alpha(q)-w}; \hspace{20mm} b(p) = \int_0^p \frac{dq}{w-\beta(q)}. \]

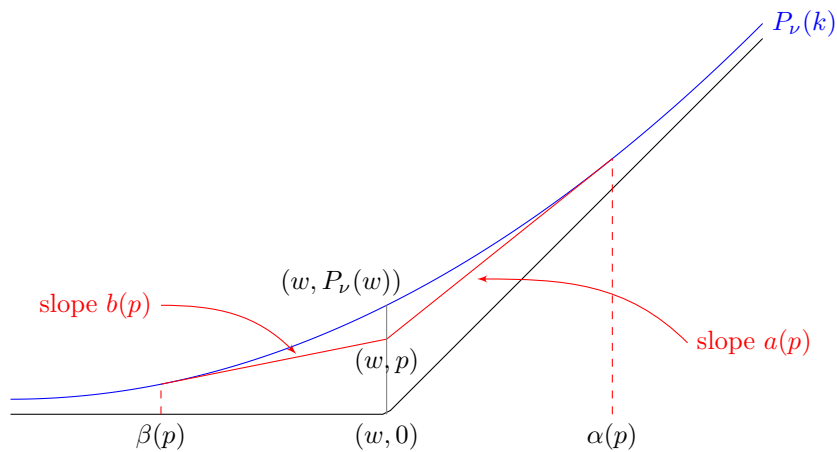
\begin{figure}[H]

\centering
\begin{tikzpicture}[scale=1,
declare function={	
diag(\x)=\x;
k1=5;
P1(\x)=(\x-k1)*(\x>k1);
    	a=0.05;
	c1=1;
	c2=0.2;
	quad1(\x)=a*(\x^2)+c1;
	quad2(\x)=a*(\x^2)+c2;
	slope=0.2;
	xs(\x)=\x/(2*a);
	my=(k1-xs(slope))*slope+quad2(xs(slope));
	d=4*k1*k1*a*a-4*a*(my-c2);
	xr=k1+sqrt(d)/(2*a);
}]

           %payoff
               \draw[name path=payoff, black] plot[ domain=0:10, samples=100] (\x,{P1(\x)});

                %P_mu
               \node (p1)[right,blue] at (10,{quad2(10)}) {$P_\nu(k)$};
                \draw[name path=payoff, blue] plot[ domain=0:10, samples=100] (\x,{quad2(\x)});

                \draw[name path=tan, red] ({xs(slope)},{quad2(xs(slope))}) -- (k1,my);
                \draw[name path=tan, red] (k1,my) -- (xr,{quad2(xr)});

              %ticks

             \draw[red, dashed] ({xs(slope)},0) -- ({xs(slope)},{quad2(xs(slope))});
             \node[below] at ({xs(slope)},0) {$\beta(p)$};

              \draw[red, dashed] ({xr},0) -- (xr,{quad2(xr)});
             \node[below] at (xr,0) {$\alpha(p)$};

              \draw[gray] (k1,0) -- (k1,{quad2(k1)});
             \node[below] at (k1,0) {$(w,0)$};

               \node[below] at (k1,my) {$(w,p)$};
               \node[above] at (k1-0.6,{quad2(k1)}) {$(w,P_\nu(w))$};

               \draw[-latex',red] ({xs(slope)},{quad2(k1)}) to[out=0 in=180] (k1-1.2,0.8);
               \node[red, left] at ({xs(slope)},{quad2(k1)}) {$\text{slope } b(p)$};

               \draw[latex'-,red] (k1+1.2,01.8)  to[out=0 in=180] (xr+1,{quad2(k1)-0.5});
               \node[red, right] at (xr+1,{quad2(k1)-0.5}) {$\text{slope } a(p)$};

\end{tikzpicture}
\caption{The definitions of $\alpha$, $\beta$, $a$ and $b$. $\Upsilon(p)$ is the difference in the slopes of the tangents to $P_\nu(k)$ which pass through $(w,p)$. }
\label{fig:alphabetaab}
\end{figure}

Let $\Upsilon:[0,P(w)] \mapsto [0,1]$  be given by $\Upsilon(p) = a(p) - b(p)$. Then $\Upsilon(0)=1$ and $\Upsilon(P(w)) = \nu ( \{ w \} )$. %Set $\Upsilon(P(0)) = 0$.
$\Upsilon$ is a decreasing, concave function which is absolutely continuous on $[0,P(w))$. We can define an inverse $\Upsilon^{-1}:[0,1] \rightarrow [0,P(w)]$ provided we set $\Upsilon^{-1}(q)=1$ for $q \leq \nu(\{w\})$.  Where $\alpha$ and $\beta$ are continuous we have $\Upsilon'(p) = - \frac{1}{\alpha(p)-w} - \frac{1}{w-\beta(p)}$.
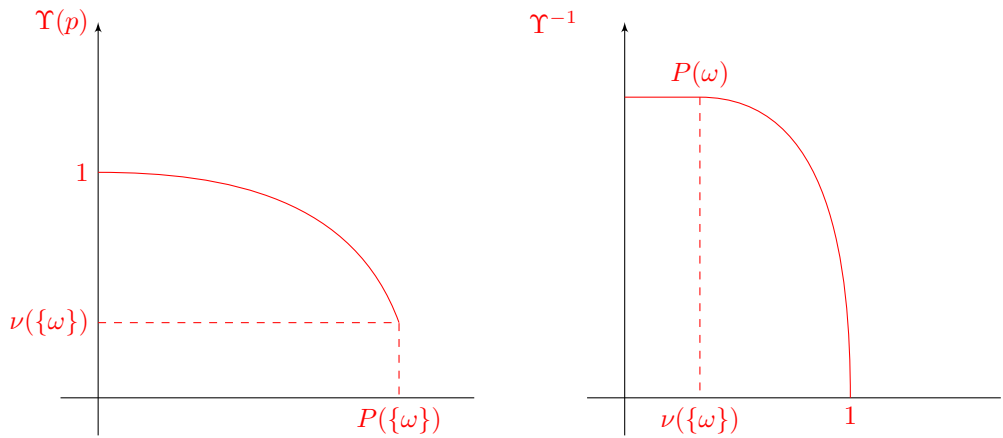
\begin{figure}[H]

\centering
\begin{tikzpicture}[scale=1,
declare function={	
diag(\x)=\x;
k1=5;
d=0.5;
k1x=1.5;
x1=1;
x2=2;
x3a=3;
x3=3.5;
x4=5.5;
x5=10;
P1(\x)=(1-d)*(\x-k1)*(\x>k1);
    	a=(1-d)*0.05;
	c1=1;
	c2=0.2;
	quad1(\x)=a*(\x^2)+c1;
	quad2(\x)=a*(\x^2)+c2;
	slope=0.02;
	xs(\x)=\x/(2*a);
	my=(k1x-xs(slope))*slope+quad2(xs(slope));
		xr=k1x+sqrt(d)/(2*a);
	d1=4*k1x*k1x*a*a+4*a*c2;
	xr1=k1x+sqrt(d1)/(2*a);
}]

%axis
               \draw[name path=payoff, black,-latex'] (-5,-.5)--(-5,5);
               \draw[name path=payoff, black,-latex'] (2,-.5)--(2,5);
               \draw[name path=payoff, black] (-5.5,0)--(0,0);
               \draw[name path=payoff, black] (1.5,0)--(7,0);
    %left
             \draw[red] (-5,3)  to[out=0, in=110] (-1,1) ;
             \draw[red, dashed] (-5,1) -- (-1,1) -- (-1,0) ;

              \node[left,red] at (-5,5) {$\Upsilon(p)$};
                \node[left,red] at (-5,3) {$1$};
             \node[left,red] at (-5,1) {$\nu(\{\omega\})$};
             \node[below,red] at (-1,0) {$P(\{\omega\})$};
   %right
              \draw[red] (2,4) -- (3,4) to[out=0, in=90] (5,0) ;
             \draw[red, dashed] (3,4) -- (3,0) ;

              \node[left,red] at (1.5,5) {$\Upsilon^{-1}$};
                \node[above,red] at (3,4) {$P(\omega)$};
             \node[below,red] at (3,0) {$\nu(\{\omega\})$};
             \node[below,red] at (5,0) {$1$};

\end{tikzpicture}

\caption{Sketch of $\Upsilon$ and $\Upsilon^{-1}$. }
\label{fig:Upsilon}
\end{figure}

Define $S:(0,1) \mapsto \R$ by $S(u) = (\alpha \circ \Upsilon^{-1})(u)$ and $R:(0,1) \mapsto \R$ by $R(u) = (\beta \circ \Upsilon^{-1})(u)$. %Since $\alpha$ and $\beta$ are right-continuous, $S$ and $R$ are left-continuous.

\begin{rem}
If $\nu$ does not charge an open interval $A \subset (w,\infty)$, then $P$ is linear on $A$. Then $\alpha$ jumps over this set and $S$ does not take values in $A$. Similarly if $\nu$ does not charge an open interval $B \subset (-\infty,w)$ then $R$ jumps over this interval.
\label{rem:jump}
\end{rem}
\begin{rem}
By construction, $\alpha$ and $\beta$ are both right-continuous. Since $\Upsilon^{-1}$ is continuous and decreasing, it follows that $R$ and $S$ are left-continuous. Moreover, %$\alpha$ and $\beta$ are defined to be tangents to $P_\nu$. It follows that
$\lim_{u\to1}R(u)=\ell_\nu$ and $\lim_{u\to1}S(u)=r_\nu$.
\label{rem:limit}
\end{rem}

Let $Y$ be defined by \eqref{eq:YUVdef} in  Theorem~\ref{thm:atoms}. Note that since $\mu$ is a point mass $G(u)=w$ for all $u \in (0,1)$.
%Let $U \sim U[0,1]$ and let $Y$ be defined by $Y \in \{ R(U), S(U) \}$ with $\Prob(Y=S(u)|U=u) = \frac{|R(u)|}{S(u)-R(u)}$. Note, for $u< \nu(\{0\})$, %$\Lambda^{-1}(u) = P(0)$ and $\alpha \circ \Lambda^{-1}(u) = 0$. We have $S(0)=R(0) =0$ for $u < \nu(\{0\})$.

\begin{lem} Suppose $U,V$ are independent uniform random variables.
Then $Y(U,V)$ has law $\nu$.
\label{lem:test}
\end{lem}

\begin{proof}
Let $\phi$ be a test function: a continuously differentiable function with support contained in $[w+\epsilon, w+\epsilon^{-1}]$ for some $\epsilon \in (0,1)$. We will show that $\E[\phi(Y)] = \int \phi(y) \nu(dy)$. We can prove a similar result for test functions $\psi$ with support in $[w-\epsilon^{-1}, w-\epsilon]$. It follows that $\sL(Y)=\nu$.

By construction
\begin{eqnarray*}
\E[\phi(Y)] & = & \int_0^1 du \frac{w-R(u)}{S(u)-R(u)} \phi(S(u)) \\
& = & \int_0^1 du \frac{w-\beta \circ \Upsilon^{-1}(u)}{\alpha \circ \Upsilon^{-1}(u)-\beta \circ \Upsilon^{-1}(u)}\phi(\alpha \circ \Upsilon^{-1}(u)) \\
& = & \int_0^{P(w)} dp |\Upsilon'(p)| \frac{w-\beta(p)}{\alpha(p)-\beta(p)}\phi(\alpha(p)).
\end{eqnarray*}

But $\Upsilon'(p) = - \frac{\alpha(p) - \beta(p)}{(\alpha(p)-w)(w-\beta(p))}$. Thus, writing $\psi(y) = \frac{\phi(y)}{(y-w)}$ and using the fact that $\alpha^{-1}(y)=P(y)-(y-w)P'(y)$ except at the countably many points where $\alpha^{-1}$ is multi-valued, %{\bf Need to explain why $\alpha^{-1}(y) = P(y)-yP'(y)$}
\begin{eqnarray*}
\E[\phi(Y)] & = & \int_0^{P(w)} dp \frac{\phi(\alpha(p))}{\alpha(p)-w} =  - \int_w^{\infty} d (\alpha^{-1}(y)) \psi(y)
 =  \int_w^\infty [ P(y)-(y-w)P'(y) ] \psi'(y) dy \\
& = & - \int_w^\infty P'(y) [ \psi(y) + (y-w) \psi'(y) ] dy = - \int_w^\infty P'(y) \phi'(y) dy = \int \phi(y) \nu(dy) .
\end{eqnarray*}

%\begin{eqnarray*}
%\E[\phi(Y)] & = & \int_0^{P(w)} dp \frac{\phi(\alpha(p))}{\alpha(p)-w} \\
%& = & - \int_w^{\infty} d (\alpha^{-1}(y)) \frac{\phi(y)}{y-w} \\
%& = & - \left[ \alpha^{-1}(y) \frac{\phi(y)}{y-w} \right]^\infty_w + \int_w^\infty \alpha^{-1}(y) \frac{d}{dy} \left( \frac{\phi(y)}{y-w} \right) dy \\
%& = & \int_w^\infty [P(y)-(y-w)P'(y)] \left[\frac{\phi'(y)}{y-w} - \frac{\phi(y)}{(y-w)^2}\right] dy \\
%& = & -\int_w^\infty P'(y) \phi'(y) dy + \int_w^\infty \frac{P(y) \phi'(y)}{y-w} dy - \int_w^\infty \frac{\phi(y)}{(y-w)^2}[P(y)-(y-w)P'(y)] dy
%\end{eqnarray*}
%where we use the fact that $\alpha^{-1}(y)=P(y)-(y-w)P'(y)$, except at the countably many points where $\alpha^{-1}$ is multi-valued. But
%\[ \int_w^\infty \phi(y) \nu(dy) = [\phi(y)P'(y) ]_w^\infty -  \int_w^\infty P'(y) \phi'(y) dy \]
%whilst
%\[ \int_w^\infty \phi(y)\frac{[P(y)-(y-w)P'(y)]}{(y-w)^2} dy = \int_w^\infty dy \; \phi(y) \frac{d}{dy} \left( - \frac{P(y)}{y-w} \right) = \left[ - \phi(y) \frac{P(y)}{y-w} \right]_w^\infty + \int_w^\infty \frac{P(y)}{y-w} \phi'(y) dy .\]
Hence $\E[\phi(y)] = \int \phi(y) \nu(dy)$.
\end{proof}

\begin{rem}
If $\alpha$ and $\beta$ are strictly monotonic at $\Upsilon^{-1}(u)$, then conditional on $U \leq u$, $Y$ has law $\nu$ conditioned to take values in
$[\beta \circ \Upsilon^{-1}(u), \alpha \circ \Upsilon^{-1}(u)]$. Necessarily, $\nu([\beta\circ \Upsilon^{-1}(u), \alpha \circ \Upsilon^{-1}(u)])=u$.

If there is an atom of $\nu$ at $\beta \circ \Upsilon^{-1}(u)$ or $\alpha \circ \Upsilon^{-1}(u)$ then we can choose appropriate masses $\underline{\lambda}_u$ and
$\overline{\lambda}_u$ such that
$\nu((\beta\circ \Upsilon^{-1}(u), \alpha \circ \Upsilon^{-1}(u))) + \underline{\lambda}_u \delta_{\beta \circ \Upsilon^{-1}(u)} + \overline{\lambda}_u \delta_{\alpha \circ \Upsilon^{-1}(u)}$ has total mass $u$ and mean $w$. We must have
$0 \leq \underline{\lambda}_u \leq \nu( \{ \beta \circ \Upsilon^{-1}(u) \} )$ and $0 \leq \overline{\lambda}_u \leq \nu( \{ \alpha \circ \Upsilon^{-1}(u) \} )$.

On $U \leq u_1$ let $Y^{u_1}=Y^{u_1}(U,V)$ be constructed as in \eqref{eq:YUVdef}. On $U > u_1$, let $Y^{u_1}$ be in a graveyard state $\Delta$. Then $\sL(Y^{u_1}) = \nu_{u_1} + (1-u_1) \delta_{\Delta}$ where $\nu_{u_1}$ is a measure on $[R(u_1),S(u_1)]$ with total mass $u_1$ and mean $w$. In particular, $\nu_{u_1} = \nu$ on $(R(u_1),S(u_1))$, $\nu_{u_1} \leq \nu$ on $\{ R(u_1),S(u_1) \}$ and $\nu_{u_1}=0$ on $[R(u_1),S(u_1)]^C$.
 %(unless $\nu$ has atoms at $R(u_1)$ and $S(u_1)$).
\label{rem:partway}
\end{rem}

\subsection{The case where $\mu$ consists of a finite number of atoms}
\label{ssec:finiteatoms}
Suppose $\mu = \sum_{i=1}^{N} \lambda_i \delta_{x_i}$ where $x_1 < x_2 \ldots < x_N$ with $\lambda_i > 0 $ and $\sum_{i=1}^N \lambda_i = 1$. Suppose $\nu$ is an arbitrary probability measure satisfying the convex order condition $\mu \leq_{cx} \nu$.

For $0 \leq p \leq P_\nu(x_1)$ we can construct $\alpha, \beta, a$ and $b$ as in \eqref{eq:alphabetadef} and \eqref{eq:abdef} (but relative to $x_1$ rather than the mean $w$) and set $\Upsilon=a-b$. %Section~\ref{ssec:pointmass}, see Figure~\ref{fig:x1},
For example, $\alpha(p)= \arginf_{k>x_1} \frac{P_\nu(k)-p}{k-x_1}$ and $a(p) = \inf_{k>x_1} \frac{P_\nu(k)- p}{k-x_1}$. Note that $\Upsilon(0) = \Lambda_1 := \inf_{x > x_1} \frac{P_\nu(x)}{x-x_1}$ and since $P_{\nu}(x) \geq P_\mu(x) \geq \lambda_1(x-x_1)$ we have $\Lambda_1 \geq \lambda_1$. The inverse $\Upsilon^{-1}$ can be defined on $[0,\Lambda_1]$, but we are only interested in $\Upsilon^{-1}$ over the interval $[0,\lambda_1]$. Using $\Upsilon^{-1}$ and the construction of the previous section we can define $S=\alpha \circ \Upsilon^{-1}:(0,\lambda_1] \mapsto [x_1,\infty)$  and $R=\beta \circ \Upsilon^{-1}:(0,\lambda_1] \mapsto (-\infty,x_1]$ with $S$ increasing and $R$ decreasing.

\begin{figure}[H]

\centering
\begin{tikzpicture}[scale=1,
declare function={	
diag(\x)=\x;
k1=5;
k1x=1.5;
x2=3.3;
x3=6;
x4=9;
P1(\x)=(\x-k1)*(\x>k1);
    	a=0.05;
	c1=1;
	c2=0.2;
	quad1(\x)=a*(\x^2)+c1;
	quad2(\x)=a*(\x^2)+c2;
	slope=0.02;
	xs(\x)=\x/(2*a);
	my=(k1x-xs(slope))*slope+quad2(xs(slope));
	d=4*k1x*k1x*a*a-4*a*(my-c2);
	xr=k1x+sqrt(d)/(2*a);
	d1=4*k1x*k1x*a*a+4*a*c2;
	xr1=k1x+sqrt(d1)/(2*a);
}]

           %payoff
               \draw[name path=payoff, black] plot[ domain=0:10, samples=100] (\x,{P1(\x)});

                %P_mu
               \node (p1)[right,black] at (10,{quad2(10)}) {$P_\nu(k)$};
                \draw[name path=payoff, black] plot[ domain=0:10, samples=100] (\x,{quad2(\x)});

                %tangents
                \draw[name path=tan, red, thick] ({xs(slope)},{quad2(xs(slope))}) -- (k1x,my);
                \draw[name path=tan, red, thick] (k1x,my) -- (xr,{quad2(xr)});

              %ticks
             \draw[red, dashed] ({xs(slope)},-1) -- ({xs(slope)},{quad2(xs(slope))});
             \node[below] at ({xs(slope)},-1) {$\beta(p)$};

              \draw[red, dashed] ({xr},-1) -- (xr,{quad2(xr)});
             \node[below] at (xr,-1) {$\alpha(p)$};

             \draw[gray] (k1x,0) -- (k1x,{quad2(k1x)});
             \node[below, black!30!green] at (k1x,0) {$x_1$};

             \draw[gray] (k1x,0) -- (xr1,{quad2(xr1)});
             \draw[latex'-,gray] (xr+0.45,0.6)  to[out=300, in=180] (k1+3,0.2);
             \node[right,gray] at (k1+3,0.2) {$\text{slope }\Lambda_1$};

             \draw[black!30!green, thick] (k1x,0) -- (x2,0.11) -- (x3,1.5)  -- (x4,{P1(x4)});
             \node[below, black!30!green] at (x4+0.4,{P1(x4)}) {$P_\mu(k)$};

             \draw[latex'-, black!30!green] (k1x+0.6,-.1)  to[out=300, in=180] (k1+1,-0.5);
             \node[right, black!30!green] at (k1+1,-0.5) {$\text{slope }\lambda_1$};

\end{tikzpicture}
\caption{Calculation of $\alpha$, $\beta$, $a$ and $b$ in this case}
\label{fig:x1}
\end{figure}
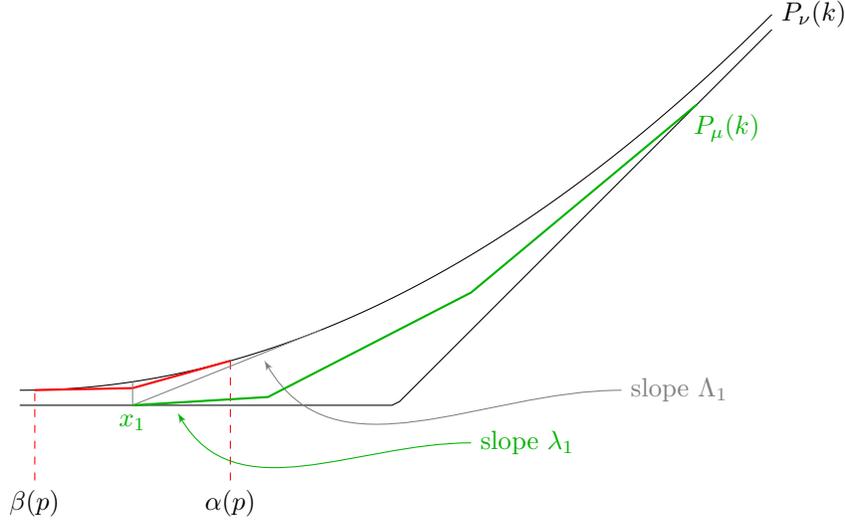

By the final comments in Remark~\ref{rem:partway}, the construction of $R$ and $S$ on $(0,\lambda_1]$ is such that if $Y$ is constructed as in
\eqref{eq:YUVdef}, then on $U \leq \lambda_1$ we find $Y$ has law $\nu_{\lambda_1}$, where $\nu_{\lambda_1} = \nu$ on $(R(\lambda_1),S(\lambda_1))$ and $\nu_{\lambda_1} \leq \nu$ on $\{ R(\lambda_1),S(\lambda_1) \}$.

We now claim that $\tilde{\mu}_1 := \mu - \lambda_1 \delta_{x_1} = \sum_{i=2}^N \lambda_i \delta_{x_i}$ and $\tilde{\nu}_1 = \nu - \nu_{\lambda_1}$ are in convex order. By construction $\nu_{\lambda_1}$ has mass $\lambda_1$ and barycentre $x_1$. Hence $\tilde{\mu}_1$ and $\tilde{\nu}_1$ also have the same total mass and barycentre.
\begin{lem}  $\tilde{\mu}_1 \leq_{cx} \tilde{\nu}_1$.
\label{lem:cx}
\end{lem}

\begin{proof}
Let $\hat{\nu} = \lambda_1 \delta_{x_1} + \tilde{\nu}_1$. Since $\lambda_1 \delta_{x_1} \leq_{cx} \nu_{\lambda_1}$ we have $\hat{\nu} \leq_{cx} \nu$.
Also $P_\mu(k) \leq P_{\hat{\nu}}(k)$. To see this note that $P_{\hat{\nu}}$ is continuous everywhere and linear on intervals $[R(\lambda_1),x_1]$ and $[x_1, S(\lambda_1)]$, whereas $P_\mu$ is continuous and convex on $[R(\lambda_1),S(\lambda_1)]$. Moreover, $P_\mu(R(\lambda_1))=0 \leq P_{\hat{\nu}}(R(\lambda_1))$, $P_\mu(x_1)=0 \leq P_{\hat{\nu}}(R(x_1))$ and
$P_\mu(S(\lambda_1)) \leq P_{\nu}(S(\lambda_1)) = P_{\hat{\nu}}(S(\lambda_1))$.
%{\bf Proof by diagram?}
Hence $P_{\tilde{\mu}}(k) + \lambda_1(x_1-k)^+ = P_{\mu}(k) \leq P_{\hat{\nu}}(k) = P_{\tilde{\nu}_1}(k) + \lambda_1(x_1-k)^+$ and it follows that
$P_{\tilde{\mu}}(k) \leq P_{\tilde{\nu}_1}(k)$ as required. %See Figure~\ref{fig:modifiedP}.
\end{proof}

We have constructed $(R,S)$ on $(0,\lambda_1]$ with $S$ increasing and $R$ decreasing in such a way that the point mass at $x_1$ is mapped to $\nu_{\lambda_1}$. It remains to embed $\tilde{\nu}_1$ starting from $\tilde{\mu}_1$. Note that by Remark~\ref{rem:partway}, $\tilde{\nu}_1$ places no mass on $(R(\lambda_1), S(\lambda_1))$.

As a next step we embed the atom $\lambda_2 \delta_{x_2}$ of $\tilde{\mu}_1$ in $\tilde{\nu}_1$. $x_2$ is the lowest location of an atom in $\tilde{\mu}_1$ so we can use the same algorithm as before. In this way, for $\lambda_1 < u \leq \lambda_1+\lambda_2$ we construct $S$ increasing with $S(\lambda_1 +) \geq S(\lambda_1-) \vee x_2$ and $R$ decreasing with $R(\lambda_1+) \leq x_2$. By Remark~\ref{rem:jump}, $R$ jumps over the interval $(R(\lambda_1), S(\lambda_1))$. We conclude  that for $0<u<v<\lambda_1+\lambda_2$, $R(v) \notin (R(u),S(u))$.

Thereafter, we proceed inductively on the number of atoms which have been embedded.
The initial law is a sub-probability $\tilde{\mu}_k = \sum_{k+1}^N \lambda_i \delta_{x_i}$ which we want to map to a target law $\tilde{\nu}_k$ where $\tilde{\mu}_k \leq_{cx} \tilde{\nu}_k$ and $\tilde{\nu}_k \leq \nu$.
Since $\mu$ consists of a finite number of atoms the construction terminates. Moreover the random variable $Y$ we construct in this way has law $\nu$ and $R$ and $S$ have the properties in \eqref{eq:RSproperties}. It follows that we have proved Theorem~\ref{thm:atoms} in the case where $\mu$ consists of a finite number of atoms.

\subsection{The martingale coupling and its inverse as maps}
Given $\nu$ centred, (and $\mu = \delta_0$) we saw in Section~\ref{ssec:pointmass} how to construct $R:(0,1) \mapsto \R_-$ and $S:(0,1) \mapsto \R_+$ such that $Y=Y(U,V)$ has law $\nu$ where $Y$ is given by $Y(u,v)=0$ if $S(u)= 0$ and
\begin{equation} Y(u,v) = R(u) I_{ \{ v \leq \frac{S(u)}{S(u)-R(u)} \} } + S(u)I_{ \{ v > \frac{S(u)}{S(u)-R(u)} \} } \label{eq:Ydefptmass} \end{equation}
otherwise.

Let $\sP^0(\R)$ denote the set of centred probability measures on $\R$.
Let $\sV^1$ denote the set of pairs of functions $R,S$ with $R:(0,1) \rightarrow \R_-$ and $S:(0,1) \rightarrow \R_+$, let $\sV^1_{Mon}$ denote the subset of $\sV^1$ for which $R$ is decreasing and $S$ is increasing, and let $\sV^1_{Int}$ denote the subset of $\sV^1$ such that $I(R,S)<\infty$ where
\[ I(f,g) = \int_0^1 du \frac{|f(u)|g(u)}{g(u)-f(u)} I_{ \{g(u) > 0 \} }. \]
Finally, let $\sV^1_{Mon,Int} = \sV^1_{Mon} \cap \sV^1_{Int}$.

The construction in Section~\ref{ssec:pointmass} can be considered as a pair of maps
\begin{eqnarray*}
 & \sQ^1  :  \sP^0(\R) \mapsto \sV^1_{Mon,Int} & \\
 & \sR^1  :  \sV^1_{Mon,Int} \mapsto \sP^0(\R) &
\end{eqnarray*}
Note that $\E[|Y|] = 2I(R,S)$ which can be shown using the ideas in the proof of Lemma~\ref{lem:test} to be equal to $2P_{\nu}(0)$. Moreover, under $I(R,S)<\infty$ we have $\E[Y]=0$.

Note that if we take $(R,S) \in \sV^1_{Mon} \setminus \sV^1_{Mon,Int}$ then we can still define $Y$ via \eqref{eq:Ydefptmass} but $\sL(Y)$ will not be integrable. Then $M$ given by $M_1=0$, $M_2 = Y$ is a local martingale, but not a martingale.

Section~\ref{ssec:finiteatoms} extends these results from initial laws which consist of a single atom to finite combinations of atoms. Let $\sP^0_F(\R)$ be the subset of $\sP^0(\R)$ for which the measure consists of a finite set of atoms and let $\sC_F = \{ (\zeta, \chi) : \zeta \in \sP^0_F(\R), \chi \in \sP^0(\R) ; \zeta \leq_{cx} \chi \}$. Let
\begin{eqnarray*}
 \sV = \{ (R,G,S) ; R:(0,1) \rightarrow \R, G:(0,1) \rightarrow \R, S:(0,1) \rightarrow \R; R(u) \leq G(u) \leq S(u);  & \; & \\
  \int_0^1 |G(u)| du < \infty, \int_0^1 G(u) du = 0 \} & . &
\end{eqnarray*}
Consider now the subsets
\begin{eqnarray*}
\sV_F & = & \{ (R,G,S) \in \sV : \mbox{$G$ non-decreasing and takes only finitely many values} \}, \\
\sV_{Mon} & = & \{ (R,G,S) \in \sV : \mbox{\eqref{eq:RSproperties} holds} \}, \\
\sV_{Int} & = & \{ (R,G,S) \in \sV : I(R,G,S) < \infty \};
\end{eqnarray*}
where $I(R,G,S) = \int_0^1 du \frac{(S(u)-G(u))(G(u)-R(u))}{S(u)-R(u)} I_{ \{S(u) > G(u) \} }$, and consider also intersections of these subsets, for example $
\sV_{Mon,Int} = \sV_{Mon} \cap \sV_{Int}$. In Section~\ref{ssec:finiteatoms} we constructed a map $Q: \sC_F \rightarrow \sV_{F, Mon}$ which we write as
$Q(\zeta,\chi) = (R_{(\zeta,\chi)}, G_\zeta, S_{(\zeta,\chi)})$. Indeed, since $\chi \in L^1$ and since
$\E[|Y-X|] \leq \E[|X|] + \E[|Y|] < \infty$ we have that
$\E[|Y-X|] = 2 I (R_{(\zeta,\chi)}, G_\zeta, S_{(\zeta,\chi)})$, so that we actually have a map $Q: \sC_F \rightarrow \sV_{F,Mon,Int}$. Conversely, the arguments after Lemma~\ref{lem:cx}  show that \eqref{eq:YUVdef} defines a inverse map $\sR:\sV_{F,Mon,Int} \rightarrow \sC_F$.

Note that given any element $(R,G,S)$ of $\sV$ we can define the map $\sR: \sV \rightarrow \sP^0(\R) \times \sP(\R)$ via $\sR(R,G,S) = (\sL(X(U)),\sL(Y(U,V)))$ where $Y(U,V)$ is as given in the statement of Theorem~\ref{thm:atoms}. We will make no further use of this idea, but different properties of $(R,G,S)$ will lead to different (local)-martingale couplings. The embedding of Hobson and Neuberger~\cite{HobsonNeuberger:12} is of this type. In the Hobson and Neuberger embedding $R$ and $S$ are both increasing.

\subsection{The case of general integrable $\mu$}
We assume $\mu$ is centred at zero, but the general case follows by translation.

%Let $\sP(\R)$ be the set of centred probability measures and let $\sP_F(\R)$ be the subset of $\sP(\R)$ for which the measure consists of a finite number of atoms. Let $\sC = \{ (\zeta,\chi) : \zeta \in \sP(\R), \chi \in \sP(\R) ; \zeta \leq_{cx} \chi \}$ and let $\sC_F = \{ (\zeta,\chi) : \zeta \in \sP_F(\R), \chi \in \sP(\R) ; \zeta \leq_{cx} \chi \}$ be sets of pairs of probability measures in convex order.
%Let $\sV = \{(R,G,S) : [0,1] \mapsto \R \mbox{ such that $G$ is increasing, } \int_0^1 G(u) du=0 \mbox{ and \eqref{eq:RSproperties} holds} \}$.

%\begin{quote}
%{\bf do we need $\int I_{S \neq G} \frac{(S-G)(G-R)}{S-R)} du < \infty$?}
%\end{quote}
%From the analysis of the previous section, we have an map $\sQ: \sC_F \rightarrow \sV$ such that $\sQ:(\zeta,\chi) \mapsto (R_{\zeta,\chi}, G_{\zeta}, S_{\zeta,\chi})$, and an extended martingale coupling $\sR: \sV \mapsto \sC$. Note that $\sR$ is an inverse to $\sQ$ so that if $(\zeta,\chi) \in \sC_F$ then $\sR(\sQ(\zeta, \chi)) = (\zeta, \chi)$.
%Our goal is to extend the mapping $\sQ$ from $\sC_F$ to $\sC$. Then, given general centred measures $\zeta$ and $\chi$ in convex order we can define a trio $(R_{\zeta,\chi}, G_{\zeta}, S_{\zeta,\chi})$ such that $(R_{\zeta,\chi}, G_{\zeta}, S_{\zeta,\chi})$ define an extended left-curtain coupling. Moreover,the extended martingale coupling $(R_{\zeta,\chi}, G_{\zeta}, S_{\zeta,\chi})$ provides a construction of martingale coupling between $\zeta$ and $\chi$.

Our goal in this section is to extend the map $\sQ : \sC_F \rightarrow \sV_{F,Mon,Int}$ with inverse $\sR$ to a map $\sQ : \sC \rightarrow \sV_{Mon,Int}$ where $\sC = \{ (\zeta, \chi): \zeta \in \sP^0(\R), \chi \in \sP^0(\R); \zeta \leq_{cx} \chi \}$.
For $\mu$ a general centred probability measure and $\nu$ a centred target measure with $\mu \leq_{cx} \nu$ we construct a sequence $(\mu_n)_{n \geq 1}$ of approximations of $\mu$ by elements of $\sP^0_F(\R)$. For each $\mu_n$ we can construct a triple $(R_n,G_n,S_n)$. We show that $(R_n,G_n,S_n)_{n \geq 1}$ converge to a limit $(R,G,S)$ first on the rationals and then (almost surely) on $(0,1)$.
Convergence of $G_n$ and $S_n$ is straightforward, but convergence of $R_n$ is more subtle, and indeed we only have convergence on $\{u: S(u)>G(u) \}$. Finally we show that $\sR(R,G,S) =(\mu, \nu)$ so that the trio $(R,G,S)$ defines a martingale coupling between $\mu$ and $\nu$.

Let $\{q_1, q_2 \ldots \}$ be an enumeration of $\Q \cap (0,1)$. Then $\{ S_{n}(q_1) \}_{n \geq 1}$ converges down a subsequence $n_{k_1}$ to a limit $S_\infty(q_1) := \lim_{k_1 \uparrow \infty} S_{n_{k_1}}(q_1)$. Down a further subsequence if necessary we have that $S_{n_{k_2}}(q_2)$ converges to $S_\infty (q_2)$. Proceeding inductively, we have by a diagonal argument (see, for example, Billingsley~\cite{Billingsley:2013}) that there is a subsequence $(m_1, m_2, \ldots)$ such that $\{ S_{m_k} \}_{k \geq 1}$ converges to $S_\infty$ at every rational $q \in \Q \cap (0,1)$. This limit is non-decreasing.
% and finite valued by Lemma~\ref{lem:bound}.

Our first result shows that any limit of $S_n$ is finite valued. Since the ideas behind the proof are not relevant to the arguments of this section the proof is postponed to  Appendix~\ref{app:proofs}.

\begin{lem}
%\label{lem:bdsonSR}
Let $\mu_n \uparrow_{cx} \mu$. %Let $\bar{\lambda}\geq 0$ be the size of any atom at the upper limit of the support of $\mu$.
Then $\limsup S_n(u) \leq J(u)$ for some function $J=J_{\mu,\nu}:(0,1) \mapsto (-\infty,\infty)$.
\label{lem:boundS}
\end{lem}

We want to extend the domain from the rationals to $(0,1)$. To this end define $S(u) = \lim_{q_j \uparrow u} S_\infty(q_j)$. This limit is well defined (and non-decreasing) by the monotonicity of $S_\infty$. Then from the monotonicity of $S$ we conclude that $S$ has only countably many discontinuities. Note that, by definition, $S$ is left-continuous.
%and in Lemma~\ref{lem:converge} below we show that $S_n(u) \rightarrow S(u)$ Lebesgue almost surely on $[0,1]$.

We can construct $G$ from $\{ G_n \}$ in an identical fashion. In this case the finiteness of the limit follows from the tightness of the singleton $\{ G \}$.
Moreover, since $G_n \leq S_n$ by construction, we have $G_\infty \leq S_\infty$ and $G \leq S$. Again, the increasing limit $G$ has at most countably many discontinuities and is left-continuous.

Define $\sN_S = \{ u : S_n(u) \not\rightarrow S(u) \}$ and $\sN_G = \{ u : G_n(u) \not\rightarrow G(u) \}$
where the subscript $n$ refers to a subsequence down which $S_n$ and $G_n$ converge on rationals. %, $\sN_R = \{ u : R_n(u) \not\rightarrow R(u) ; G(u)<S(u) \}$.
Define also $\sN^\Delta_S = \{ u : S(u+) > S(u-) \}$ and $\sN^\Delta_G = \{ u : G(+) > G(u-) \}$. %, $\sN^\Delta_R = \{ u : R(u+) < R(u) ; G(u)<S(u) \}$.

\begin{lem}
%$\sN^\Delta_S$, $\sN^\Delta_G$ and $\sN^\Delta_R$ are all countable.
$\sN_S \subseteq \sN^\Delta_S$ and $\sN_G \subseteq \sN^\Delta_G$.
%and $\sN_R \subseteq \sN^\Delta_R$ and
Moreover, $\Leb(\sN_S \cup \sN_G) = 0$.
\label{lem:convergeSG}
\end{lem}

\begin{proof} %The countability of
%We first show that if $S(u+)=S(u)$ then $S_n(u) \rightarrow S(u)$.

Suppose $u$ is a continuity point of $S$. Suppose further that there is a subsequence $(n_j)_{j \geq 1}$ along which $S_{n_j}(u) > S(u)+\epsilon$. Using the continuity of $S$ at $u$ we may pick $q>u$ such that $S(q) < S(u) + \epsilon/2$. Take $q_k \in (u,q)$ with $q_k \downarrow u$. Then $S_{n_j}(q_k) \geq S_{n_j}(u) > S(u)+\epsilon > S(q_k) + \epsilon/2$. Letting $j \uparrow \infty$, $S_\infty(q_k) > S(q_k) + \epsilon/2$. Letting $k \uparrow \infty$, $S(u) \geq S(u) + \epsilon/2$ which is a contradiction.

A similar argument (without the need of continuity at $u$) shows that down any subsequence $\lim_j S_{n_j}(u) > S(u)-\epsilon$. Hence, if $S(u)=S(u+)$ then $S(u) = \lim S_{n}(u)$. Since the set of points for which $S(u+)>S(u)$ is countable we conclude that $\Leb(\sN_S)=0$.

An identical argument gives that $G(u) = \lim_n G_{n}(u)$ on $G(u+)=G(u)$ and $\Leb(\sN_G)=0$.
\end{proof}

Now consider $(R_n)_{n \geq 1}$ and the existence of a possible limit $R$. %We will not be interested in the definition of $R$ on $\{ u : G(u)=S(u) \}$ so on this set we define $R(u) = G(u)$.
By the same diagonal argument as above we can define $R_\infty : \Q \cap (0,1) \rightarrow \R$ such that on a subsequence $R_{n_k}(q) \rightarrow R_\infty(q) \in [-\infty,\infty]$ for every $q$. (From now on we work on a subsequence indexed $n$ such that $\{ S_n \}_n$, $\{ G_n \}_n$ and $\{ R_n \}_n$ converge for every $q \in \Q \cap (0,1)$.) We want to construct $R$ from $R_\infty$, but unlike in the case of $S$ or $G$ we do not have monotonicity. Note that for $q'>q$ we have $R_n(q') \notin (R_n(q),S_n(q))$ for each $n$ and this implies $R_\infty(q') \notin (R_\infty(q),S_\infty(q))$.

The following lemma shows that $R_\infty$ is finite valued, at least for $q$ such that $G(u+)<S(u)$.

\begin{lem}
Let $\mu_n \uparrow_{cx} \mu$. %Let $\bar{\lambda}\geq 0$ be the size of any atom at the upper limit of the support of $\mu$.
Then $\liminf R_n(u) \geq j(u)$ on $G(u+)<S(u)$ for some function $j=j_{\mu,\nu}:(0,1) \mapsto (-\infty,\infty)$.
\label{lem:boundR}
\end{lem}

Let $\sA = \{ u \in (0,1): G(u+) < S(u) \}$. By the above lemma $R_\infty(q) > j(q) >-\infty$ for $q \in \sA$. If $u \in \sA$ then the left continuity of $S$ implies that there exists an interval $(u - \epsilon, u] \subseteq \sA $; since every such interval must contain a rational we have that $\sA$ is a countable union of intervals.

We now show that $R_\infty$ is decreasing on each such interval. Suppose not. Then there exists $q<q'$ in the same interval $I$ with $R_\infty(q') > R_\infty(q)$. Let $v = \inf_{q'' \in \Q \cap I} \{ q'' : R_\infty(q'') > R_\infty(q) \}$. Choose $\tilde{q}_m \uparrow v$ with $\tilde{q}_m \geq q$ and $\hat{q}_n \downarrow v$ with $R_\infty(\hat{q}_n) > R_\infty(q)$. Then $R_\infty(\hat{q}_n) \not\in (R_\infty(\tilde{q}_m),S_\infty(\tilde{q}_m))$, and since $R_\infty(\hat{q}_n)>R_\infty(q) \geq R_\infty(\tilde{q}_m)$ we conclude $R_\infty(\hat{q}_n) \geq S_\infty(\tilde{q}_m)$.
Letting $n$ tend to infinity we conclude $\liminf R_\infty(\hat{q}_n) \geq S_\infty(\tilde{q}_m)$, and letting $m$ tend to infinity $\liminf_{n \uparrow \infty} R_\infty(\hat{q}_n) \geq S(v)$. However, $R_\infty(\hat{q}_n) \leq G_\infty(\hat{q}_n)$ and hence $\limsup_{n \uparrow \infty} R_\infty(\hat{q}_n) \leq G(v+) < S(v)$. These two statements are inconsistent, and hence $R_\infty$ must be decreasing on each interval of $\sA$.

Given that $R_\infty$ is decreasing on each interval of $\sA$, we can define $R$ on $\sA$ by
$R(u) = \lim_{q \uparrow u} R_\infty(q)$. Then the function $R$ is decreasing and therefore has only countably many discontinuities in
any interval of $\sA$. Away from these discontinuities, we have $R_n(u) \rightarrow R(u)$ by an argument similar to that in Lemma~\ref{lem:convergeSG}.

Define $\sB_= = \{ u \in (0,1) : G(u)=S(u) \}$ and $\sB_< = \{ u \in (0,1) : G(u)<S(u) \}$. Then $\sB_< = \sA \cup \sC$ where $\sC = \{ u \in (0,1) : G(u)<S(u) \leq G(u+)  \}$. Since $\sC \subseteq \sN^\Delta_G$, we have that $\sB_<$ and $\sA$ differ by a set of measure zero and we conclude:
\begin{lem}
$I_{ \{ u \in \sB_< \} } (R_n(u) - R(u)) \rightarrow 0$, except on a set of measure zero.
\label{lem:convergeR}
\end{lem}
Note that we cannot expect $R_n(u)$ to converge on $\sB_=$.

It remains to define $R$ on $\sB_=$ and $\sC$ in such a way that $R$ satisfies \eqref{eq:RSproperties}. On $\sB_=$ we set $R(u)=G(u)=S(u)$. For $u \in \sC$ we have by the left continuity of $S$ that there exists $\epsilon > 0$ such that $I=(u-\epsilon,u) \subset \sA$. By the same arguments as before we conclude that $R_\infty$ is decreasing on $I$ and we set $R(u) = \lim_{q \uparrow u} R_\infty(q)$. Indeed, for $u \in \sB_<$ we have $R(u) = \lim_{q \uparrow u} R_\infty(q)$. Note that for $u \in \sC$ we may have that $R(u+) > R(u)$ and it is not true in general that $R$ is decreasing on intervals contained in $\sB_<$.

Fix $u<v$.
If $u$ or $v$ is in $\sB_=$ then since we have defined $R(w)=G(w)=S(w)$ on $\sB_=$ we trivially have $R(v) \notin (R(u),S(u))$. For $u,v \in \sB_{<}$
choose sequences $\{q_m\}_m$ with $q_m < u$ and $q_m \uparrow u$ and $\{ q_l \}_l$ with $q_l \in (u,v)$ and $q_l \uparrow v$. Then $R_n(q_l) \notin (R_n(q_m),S_n(q_m))$ and hence
$R_\infty(q_l) \notin (R_\infty(q_m),S_\infty(q_m))$. Letting $l \uparrow \infty$ we have $R(v) \notin (R_\infty(q_m),S_\infty(q_m))$ and letting $m \uparrow \infty$ we have $R(v) \notin (R(u),S(u))$. Hence, $(R,G,S)$ satisfy \eqref{eq:RSproperties}.

%For $R$ we cannot expect, nor do we demand, $R_n(u) \rightarrow R(u)$ on $\sB_=$. Nonetheless, we have proved the following result:

%Instead we show that  We know that $\sB_< \cap \sA^C$ can be written as a countable union of intervals, and on each interval, $R$ has only countably many discontinuities. Away from these discontinuities $R_n(u) \rightarrow R(u)$ by a similar argument as that for $S$ above. Since $\Leb(\sA)=0$ we conclude $\Leb(\sN_R)=0$.

%\begin{proof}
%As argued above, $R$ has only countably many discontinuities on $\sB_<$, and away from these discontinuities $R_n \rightarrow R$.
%Define $\sN_R = \{ u : R_n(u) \not\rightarrow R(u) ; G(u)<S(u) \}$ and $\sN^\Delta_R = \{ u : R(u+) < R(u) ; G(u)<S(u) \}$.
%$\sN^\Delta_R$ is countable. Moreover,  $\sN_R \subseteq \sN^\Delta_R$ and $\Leb(\sN_R) = 0$.
%\end{proof}

On the space $ \{ (r,g,s) ; r \leq g \leq s \} \subseteq \R^3$ define $\Theta^x = \Theta^x(r,g,s)$ by $\Theta^x(r,g,s) = I_{ \{ r \leq x <s \} } \frac{s-g}{s-r}$ with the convention that $\Theta^x(r,g,s) = 0$ for $g=s$. In particular,
$\Theta^x(g,g,g) = 0$.

\begin{prop}
If $x$ is such that $\Leb(\{ u: S(u) = x \} \cup \{ u:R(u)=x ; S(u) > G(u) \} ) = 0$, then we have
%\begin{eqnarray}
%\lefteqn{ \int_0^1 du \left\{ I_{ \{ S_n(u) \leq x \} } + I_{ \{ R_n(u) \leq x < S_n(u) \} } \frac{S_n(u) - G_n(u)}{S_n(u)-R_n(u)} \right\} } \nonumber \\
%&\rightarrow&
% \int_0^1 du \left\{ I_{ \{ S(u) \leq x \} } + I_{ \{ R(u) \leq x < S(u) \} } \frac{S(u) - G(u)}{S(u)-R(u)} \right\}
%\label{eq:intconv}
%\end{eqnarray}
\begin{equation}
\int_0^1 du \left\{ I_{ \{ S_n(u) \leq x \} } + \Theta^x (R_n(u),G_n(u),S_n(u)) \right\} \rightarrow \int_0^1 du \left\{ I_{ \{ S(u) \leq x \} } + \Theta^x(R(u),G(u),S(u)) \right\}
\label{eq:intconv2}
\end{equation}
\label{prop:intconv}
\end{prop}

\begin{proof}
Since $S_n(u) \rightarrow S(u)$ almost surely and since $\int_0^1 du I_{ \{ S(u)=x \} } = 0$ by hypothesis, we have
$\int_0^1 du  I_{ \{ S_n(u) \leq x \} } \rightarrow \int_0^1 du  I_{ \{ S(u) \leq x \} }$ by bounded convergence.

Let $\Omega_< = \{u: S_n(u) \rightarrow S(u), G_n(u) \rightarrow G(u), R_n(u) \rightarrow R(u), G(u)<S(u) \}$ and $\Omega_= = \{u: S_n(u) \rightarrow S(u), G_n(u) \rightarrow G(u), G(u)=S(u) \}$. By Lemmas~\ref{lem:convergeSG} and \ref{lem:convergeR}, $\Leb(\Omega_< \cup \Omega_=)=1$.

Now let $\Omega^x_< = \{u: S_n(u) \rightarrow S(u) \neq x, G_n(u) \rightarrow G(u), R_n(u) \rightarrow R(u) \neq x, G(u)<S(u) \}$ and $\Omega^x_= = \{u: S_n(u) \rightarrow S(u) \neq x, G_n(u) \rightarrow G(u), G(u)=S(u) \}$. By the hypothesis on $x$ we still have that $\Leb(\Omega^x_< \cup \Omega^x_=)=1$, and by bounded convergence the result of the proposition will follow if we can show that $\Theta^x(R_n,G_n,S_n) \rightarrow \Theta^x(R,G,S)$ on
$\Omega^x_< \cup \Omega^x_=$.

This is immediate on $\Omega^x_<$. On $\Omega^x_=$ we need only note that, %treating the cases $(R_n(u) \leq x)$ and $(R_n(u)>x)$ separately,
\[ \Theta^x(R_n,G_n,S_n) = I_{ \{ R_n \leq x < S_n \} } \frac{(S_n - G_n)}{(S_n-R_n)} \leq \frac{(S_n - G_n)}{(S_n-x)} I_{ \{ S_n > x \} } \rightarrow 0 = \Theta^x(R,G,S) . \]

%Moreover, on $S(u)=G(u)$ either $R_n(u) \rightarrow G(u)=S(u)$ and $I_{ \{ R_n(u) \leq x < S_n(u) \} } \rightarrow 0$ or $R_n(u) \rightarrow R(u)<G(u)$ and $\frac{S_n(u)-G_n(u)}{S_n(u)-R_n(u)}  \rightarrow 0$.
%Then, except on a set of measure zero we have $S_n(u) \rightarrow S(u)$, $G_n(u) \rightarrow G(u)$ and
%$I_{G(u) < S(u)} (R_n(u)-R(u)) \rightarrow 0$. It follows that
%\[ \int_0^1 du I_{ \{ G_n(u)<S_n(u) \} } \frac{S_n(u) - G_n(u)}{S_n(u)-R_n(u)} \rightarrow \int_0^1 du  I_{ \{ G (u)<S(u) \}} \frac{S(u) - G(u)}{S(u)-R(u)}. \]
\end{proof}

\begin{proof}[Proof of Theorem~\ref{thm:atoms}]
All that remains to show is that $(R,G,S)$ embeds $\nu$.

There are at most countably many $x$ for which $\Leb(\{ u: S(u) = x \}) + \Leb (\{ u:R(u)=x ; S(u) > G(u) \} ) > 0$. Hence it is sufficient to prove that
$\int_0^1 du \left\{ I_{ \{ S(u) \leq x \} } + I_{ \{ R(u) \leq x < S(u) \} } \frac{S(u) - G(u)}{S(u)-R(u)} \right\} = \nu((-\infty,x])$ outside this set.
For such an $x$, \eqref{eq:intconv2} holds. Then, since $(R_n,G_n,S_n)$ embeds $\nu$ from $\mu_n$,
\begin{eqnarray*}
\lefteqn{\int_0^1 du \left\{ I_{ \{ S(u) \leq x \} } + I_{ \{ R(u) \leq x < S(u) \} } \frac{S(u) - G(u)}{S(u)-R(u)} \right\} }\\
& = & \lim_n \left\{ \int_0^1 du \left\{ I_{ \{ S_n(u) \leq x \} } + I_{ \{ R_n(u) \leq x < S_n(u) \} } \frac{S_n(u) - G_n(u)}{S_n(u)-R_n(u)} \right\}  \right\} \\
& = & \lim_n \nu((-\infty,x]) = \nu((-\infty,x])
\end{eqnarray*}
as required.
\end{proof}
We would like to thank the anonymous referee for the following idea for an alternative proof of Theorem~\ref{thm:atoms}.
\begin{rem}[Alternative construction]
Let $(\pi_{lc}^x)_{x\in\R}$ be the disintegration of $\pi_{lc}$ with respect to $\mu$, so that $\pi_{lc}(dx,dy)=\mu(dx)\pi_{lc}^x(dy)$. It follows that for any $\mu'\leq\mu$, $\pi'(dx,dy):=\mu'(dx)\pi_{lc}^x(dy)$ is again a left-curtain coupling.
 Decompose $\mu=\mu_c+\sum_n\alpha_n\delta_{x_n}$ into continuous and discrete parts, respectively. The desired representation of $\pi_{lc}$ through graphs of functions can then be obtained by pasting together the representations of $\pi_c(dx,dy):=\mu_c(dx)\pi_{lc}^x(dy)$ and $\pi_d(dx,dy):=\sum_n\alpha_n\delta_{x_n}(dx)\pi_{lc}^{x_n}(dy)$. Note that in the case of $\pi_c$, the result of Theorem~\ref{thm:atoms} follows from the original theorem of Beiglb{\"o}ck and Juillet~\cite{BeiglbockJuillet:16}, while the case of $\pi_d$ follows from the arguments given in Section~\ref{sec:atoms}.
\end{rem}

\section{Robust bounds for the American put}

Our motivation for the study of the left-curtain mapping came from a connection with the robust pricing of American puts. In robust or model-independent pricing (Hobson~\cite{Hobson:98,Hobson:survey}) the idea is that instead of writing down a model for the asset price (for example, geometric Brownian motion or a stochastic volatility model) we consider the class of all models for which the discounted asset price is a martingale and which are consistent with the prices of traded vanilla options. Then, given an exotic option which we would like to price, we search over this class of models to find the range of feasible model-based prices.

Typically the set of traded vanilla options is taken to be the set of European-style puts and calls. Given a family of European puts and calls for a fixed maturity and a continuum of strikes we can infer the law of the asset price at that maturity (under the market measure used for pricing). Given the prices of puts and calls for a sequence of maturities  we can infer the marginal distributions of the asset price, but not the joint distributions. Then, working under the bond-price numeraire, the class of asset price processes which are consistent with the prices of traded vanilla options can be identified with the class of martingales with given marginals.  The problem of finding the robust upper bound on the price of an American-style option becomes a search over consistent martingale models of the model-based price of the American option, see Neuberger~\cite{Neuberger:07}, Hobson and Neuberger~\cite{HobsonNeuberger:17} and Bayraktar et al.~\cite{BayraktarHuangZhou:15}. Crucially, the primal pricing problem can be identified with a dual hedging problem.

When the American-style option is an American put and the number of candidate exercise dates is two, Hobson and Norgilas~\cite{HobsonNorgilas:17} solve for the robust upper bound under an assumption that the law of the underlying at the first exercise date is continuous. It turns out that the consistent model for which the American put has highest price is the model associated with the left-curtain coupling of Beiglb\"ock and Juillet~\cite{BeiglbockJuillet:16}. Here we briefly explain how the results of Hobson and Norgilas extend to the atomic case, and why the atomic case is important. There is a subtlety in the case with atoms which is not present when there are no atoms, and to deal with this subtlety we need the extension of the left-curtain coupling to the atomic case as constructed in this paper.

%Let $\mu$ and $\nu$  be probability measures in convex order such that both $\mu$ and $\nu$ have mean $w$. Let $\Pi(\mu,\nu)$ be the set of martingale couplings of $\mu$ and $\nu$, so that $\pi \in \Pi(\mu,\nu)$ is a measure on $\R \times \R$ such that
%\[ \int_y \pi(dx,dy) = \mu(dx) \; \;\forall x; \hspace{10mm}
%\int_x \pi(dx,dy) = \nu(dy) \; \; \forall y; \hspace{10mm}
%\int_y (y-x) \pi(dx,dy) = 0 \; \; \forall x. \]
%We can identify $\pi \in \Pi(x,y)$ with the joint law of the discounted price process $(X_0,X_1,X_2)$ where $\sL(X_0) = \delta_w$, $\sL(X_1) = \mu$ and $\sL(X_2) = \nu$. {\bf Do we need this para?}

We are interested in pricing the American put which, in discounted units has strike $K_1$ at maturity 1 and strike $K_2$ at maturity 2, with $K_2<K_1$, see Hobson and Norgilas~\cite{HobsonNorgilas:17}. The expected payoff arising from a given joint law $\pi \in \Pi_M(\mu,\nu)$ and a given stopping rule $\tau$ taking values in $\{1,2\}$ is
\[ \phi_\pi(\tau) = \E^{\sL(X_1,X_2) \sim \pi}\left[(K_1 - X_1)^+ I_{\{ \tau = 1 \} } + (K_2 - X_2)^+ I_{\{ \tau = 2 \} } \right] \]
Here $X$ represents the discounted asset price, and is a martingale with joint law $\pi$.

For a Borel set $B$ we can let $\tau_B$ be the stopping rule $\tau_B = 1$ if $X_1 \in B$ and $\tau_B = 2$ otherwise. Then the payoff under the stopping rule $\tau_B$ is $\Phi_\pi(B) := \phi_\pi(\tau_B)$ and the American put price under the model is $\overline{\Phi}_\pi = \sup_B \Phi_\pi(B)$.

Bayraktar et al.~\cite{BayraktarHuangZhou:15}\footnote{\cite{BayraktarHuangZhou:15} contains many interesting and important results and this is just a small element of the paper} define the upper bound on the price of the American put to be
\[ {\sP}_{BHZ} = \sup_{\pi \in \Pi_M(\mu,\nu)} \overline{\Phi}_\pi = \sup_{\pi \in \Pi_M(\mu,\nu)} \sup_B \Phi_\pi(B). \]

The definition of the model-independent upper bound on the price of the American put given by Neuberger~\cite{Neuberger:07} and Hobson and Neuberger~\cite{HobsonNeuberger:17} is different.
Suppose $(\sS = (\Omega, {\mathcal F}, \mathbb F, \Prob), X=(X_0,X_1,X_2))$ is a $(\mu,\nu)$-consistent model.
The model-based price of the American put is
\[ \sA( \sS,X) = \sup_{\tau \in \sT_{1,2}(\sS)} \E^{\sS,X}[(K_\tau - X_\tau)^+ ] \]
where $\sT_{1,2}(\sS)$ is the set of all $\mathbb F$-stopping times taking values in $\{1,2\}$. Then (Neuberger~\cite{Neuberger:07}, Hobson and Neuberger~\cite{HobsonNeuberger:17}) the highest model-based price is
\begin{equation}
{\sP}_{N} = \sup_{\sS,X} \sA(\sS,X)
\label{eq:NeubergerPrimal}
\end{equation}
where the supremum is taken over $(\mu,\nu)$-consistent models.

Set $\overline{\Omega} = \R \times \R= \{\omega = (\omega_1, \omega_2)\}$, $\overline{\sF} = \sB(\Omega)$ and $(X_1(\omega),X_2(\omega)) = (\omega_1,\omega_2)$, and let $\overline{\Prob}$ be such that $\sL(X_1) = \mu$ and $\sL(X_2) = \nu$. Let $\overline{\sF}_0= \{ \emptyset, \Omega \}$, $\overline{\sF}_1 = \sigma(X_1)$ and $\overline{\sF}_2 = \sigma(X_1,X_2)$. %and let $(\overline{X}_i(\omega))_{i = 0,1,2} = (w,\omega_1, \omega_2)$.
%then $\overline{\F}$ is the natural filtration of $\overline{X}$.
If $\overline{\sS} = (\overline{\Omega},\overline{\sF},\overline{\F},\overline{\Prob})$ then $(\overline{S},\overline{X})$ is a $(\mu,\nu)$-consistent model.
%Indeed, $(\overline{S},\overline{X})$ is a simplest $(\mu,\nu)$-consistent model; note that we have not specified the joint law of $(X_1,X_2)$ so that it is not unique, but the probability space is the simplest possible.

Consistent models of the form $(\overline{S},\overline{X})$ can be identified with martingale couplings $\pi$. It follows that ${\sP}_{BHZ} \leq {\sP}_N$, the inequality following from the fact that in principle we could work on a richer probability space.
It follows from the work of Hobson and Norgilas~\cite{HobsonNorgilas:17} that if $\mu$ is continuous then the martingale coupling associated with the optimiser for either ${\sP}_{BHZ}$ or $\sP_N$ is the left-curtain coupling and ${\sP}_{BHZ} = {\sP}_N$. Our interest in extending the left-curtain mapping arose from the fact that when $\mu$ has atoms we may have ${\sP}_{BHZ} < {\sP}_N$. Then, in order to construct the optimiser for ${\sP}_N$ we need an appropriate extension of the left-curtain coupling.

\subsection{The trivial law for $\mu$}
\label{ssec:trivial}
The difference between the modelling approaches of Bayraktar et al.~\cite{BayraktarHuangZhou:15} and Hobson and Neuberger~\cite{HobsonNeuberger:17} can be illustrated most simply when $\mu = \delta_w$. Also for simplicity we assume $\nu$ has a continuous law with mean $w$.

In the framework of Bayraktar et al.~\cite{BayraktarHuangZhou:15}, since the filtration generated by $X$ is still trivial at time 1, the only choices facing the holder of the American put are either to always stop at time 1, or to never stop at time 1. The expected payoff of the American put does not depend on the martingale coupling and thus
\begin{align*}
\sP_{BHZ} =  \sup_{\pi \in \Pi_M(\mu,\nu)}\max \{ \Phi_\pi(\Omega), \Phi_\pi(\emptyset) \}  &=\sup_{\pi \in \Pi_M(\mu,\nu)} \max \{ \phi_\pi(1), \phi_\pi(2) \}\\ &= \max \left\{ (K_1-w)^+, \int (K_2-z)^+ \nu(dz) \right\}. \end{align*}

On the other hand we can construct a richer model which is $(\delta_w, \nu)$ consistent. Set $\Omega = (0,1) \times (0,1)$ and let $\Prob$ be Lebesgue measure on $\Omega$. Let $(U,V)$ be a pair of independent uniform random variables, let $(\sF_0= \{ \emptyset, \Omega \}, \sF_1= \sigma(U), \sF_2 =\sigma(U,V))$ and let $X_0=X_1=w$ and $X_2 = Y$, where $Y=Y(U,V)$ is as given in \eqref{eq:YUVdef} with $G(u) \equiv w$. Here $(R,S)$ are a pair of monotonic functions with
\begin{equation} u = \int_{R(u)}^{S(u)} \nu(dz), \hspace{30mm} 0 = \int_{R(u)}^{S(u)} (z-w) \nu(dz) . \label{eqn:trivialRS} \end{equation}
In this way we construct a  $(\mu, \nu)$-consistent model.

Under this model the value $A(u)$ of the American put under the stopping rule $\tau_u$ where $\tau_u=1$ if $U<u$ and $\tau_u=2$ otherwise is
\begin{eqnarray*}
 A(u) & = & \E[(K_1-X_1)^+ I_{\{ \tau_u = 1 \}} + (K_2 - X_2)^+ I_{\{ \tau_u = 1 \}} ] \\
 & = & (K_1-w) u + \int_{-\infty}^{R(u)} (K_2-z)^+ \nu(dz) + \int^{\infty}_{S(u)} (K_2-z)^+ \nu(dz) .
\end{eqnarray*}
It follows that ${\sP}_N \geq  \sup_{u \in [0,1]} A(u)$. (In the next section we will argue that there is equality here.) Note that ${\sP}_{BHZ} = A(0) \vee A(1)$, so that ${\sP}_{N}> {\sP}_{BHZ}$ will follow if $\sup_{u \in [0,1]} A(u) > A(0) \vee A(1)$.

For a simple example, suppose $w=1$ and $\nu = U[0,2]$; suppose $K_1=\frac{5}{4}$ and $K_2=1$. Then $R(u) = 1-u$ and $S(u)=1+u$. We have
\[ A(u) = \frac{u}{4} + \int_0^{1-u} (1-z) \frac{dz}{2} = \frac{1+u-u^2}{4} . \]
Then ${P}_N \geq \max_{u \in [0,1]} A(u) = \frac{5}{16} > \frac{1}{4} = A(0) \vee A(1) = {\sP}_{BHZ}$.

\begin{rem}
In our set-up there are two possible exercise times for the American put, denoted 1 and 2, and we construct a martingale $(X_0=w,X_1,X_2)$ to match the marginals at these times. But if $\sL(X_1) = \delta_{X_0}$ the problem can be recast as a problem for a stochastic process $\tilde{X} = (\tilde{X}_0,\tilde{X}_1)$ where $\tilde{X}_0 = X_0=X_1$ and $\tilde{X}_1 = X_2$. We also set $\tilde{\tau}=\tau-1$; then $\tilde{\tau} \in \{0,1\}$ and $\tilde{\tau}=0$ corresponds to immediate exercise. Put another way, one way to allow for immediate exercise of the American put, is to introduce an additional point (labelled 1) into the time-indexing set and to require $\sL(X_1) = \delta_{X_0}$. For this reason it is very natural for $\mu$ to have a trivial law, if we want to allow immediate exercise.
\end{rem}

\subsection{Tightness of the bound for a trivial law $\mu$}
Our goal in this section is to show that ${\sP}_N = \sup_{u \in [0,1]} A(u)$. We do this by finding an upper bound on the American put pricing problem and then showing that this bound is equal to $\sup_{u \in [0,1]} A(u)$.

Let $\psi$ be a convex function with $\psi(z) \geq (K_2-z)^+$. Let $\phi(z) = ((K_1 - z)^+ - \psi(z))^+$ and let $\theta(z) = - \psi_+'(z)$, where $\psi_+'$ is the right derivative. Then, for all $x_1$ and $x_2$
\begin{eqnarray*}
(K_1 - x_1)^+ & \leq & \phi(x_1) + \psi(x_2) + (x_2-x_1) \theta(x_1), \\
(K_2 - x_2)^+ & \leq & \phi(x_1) + \psi(x_2).
\end{eqnarray*}
It follows that for any set $B \in \sF$ and for every $\omega$,
\[ (K_1 - X_1)^+ I_B + (K_2 - X_2)^+ I_{B^C} \leq \phi(X_1) + \psi(X_2) + (X_2 - X_1) \theta(X_1) I_B . \]
In particular, if we think of $B$ as the set of scenarios on which the put is exercised at time 1 then we have that the payoff of the American put is bounded above by the sum of the European-style payoffs $\phi$ and $\psi$ and the gains from trade from a strategy which involves holding $\theta(X_1)$ units of the underlying over the time-interval $(1,2]$, provided the put was exercised at time 1. Then, for $B \in \sF_1$
\begin{eqnarray*}
\E[(K_{\tau_B} - X_{\tau_B})^+] & \leq & \E[\phi(X_1)] + \E[\psi(X_2)] \\
                               & = & \int ((K_1 - x)^+ - \psi(x))^+ \mu(dx) + \int \psi(y) \nu(dy).
\end{eqnarray*}
In our context with $\mu = \delta_w$ this simplifies to $((K_1-w)^+ - \psi(w))^+ +  \int \psi(y) \nu(dy)=: \sD(\psi)$.
Let $\sD = \inf_\psi \sD(\psi)$ (where the infimum is taken over convex $\psi$ with $\psi(z) \geq (K_2-z)^+$). $\sD$ forms an upper bound for the price of the American option under any consistent model and hence $\sP_N \leq \sD$.

Let $R$ and $S$ be defined as in Section~\ref{ssec:pointmass}. Let $P_\nu(z) = \int (z-x)^+ \nu(dx)$. Then \eqref{eqn:trivialRS} can be rewritten as $u = P'_\nu(S(u))- P'_\nu(R(u))$ together with
\begin{equation}  (S(u)-w)P'_\nu(S(u))- P_\nu(S(u)) = (R-w)P'_\nu(R(u)) - P_\nu(R(u)) . \label{eqn:trivialRS2} \end{equation}

Fix $K_2<K_1$ with $K_1>w$ and define $\Lambda_w:(-\infty, K_2 \wedge w) \times (K_1,\infty) \mapsto \R$ by
\[ \Lambda_w(r,s) = \frac{K_1-w}{s-w} - \frac{(K_2 - r)-(K_1-w)}{w-r}. \]
Since $\nu$ is continuous by assumption, $R$ and $S$ are strictly decreasing and strictly increasing, respectively. Define $u_w = \inf \{ u \in (0,1) : \mbox{$R(u)<K_2$ and $S(u)>K_1$} \}$, and for $u \in (u_w,1)$ set $\bar{\Lambda}_w(u) = \Lambda_w(R(u),S(u))$. It follows that $\bar{\Lambda}_w$ is strictly decreasing.

Suppose that $\sI_\nu = [\ell_\nu,r_\nu]$ is such that $\frac{K_1-w}{r_\nu-w} < \frac{(K_2 - \ell_\nu)-(K_1-w)}{w-\ell_\nu}$ (this will follow if $0 = \ell_\nu < w < r_\nu = \infty$ and $K_2 > K_1-w$, for example). This assumption is sufficient to guarantee that
there exists $u^* \in (u_w,1)$ such that $\bar{\Lambda}_w(u^*)=0$. Then $S^*:=S(u^*) > K_1 > K_2 > R(u^*) =: R^*$.
Also $\bar{\Lambda}_w(u^*)=0$ implies $\frac{K_1-w}{S^*-w} = \frac{K_2-R^*}{S^*-R^*}$. For the model constructed in Section~\ref{ssec:pointmass} we have
\begin{eqnarray*}
 \sup_{u \in [0,1]}A(u) \geq A(u^*) & = & (K_1-w)^+ u^* + \int_{-\infty}^{R^*}(K_2-z)^+ \nu(dz) \\
  &=& (K_1-w)[P'_\nu(S^*) - P'_\nu(R^*)] + P_\nu(R^*) + (K_2-R^*) P'_\nu(R^*).
\end{eqnarray*}

Conversely, let $\Theta = \frac{K_1-w}{S^*-w} = \frac{K_2-R^*}{S^*-R^*} = \frac{(K_2 - R^*)-(K_1-w)}{w-R^*}\in (0,1)$ and let $\psi^*(x) = \Theta(S^*-x)^+ + (1- \Theta)(R^*-x)^+$. Note that by design $\psi^*(R^*) = \Theta(S^*-R^*)=(K_2-R^*)$ so that $\psi^*(z) \geq (K_2 - z)^+$. Further, $\psi^*(w) = \Theta(S^*-w)=(K_1-w)$ so that $\phi^*(w)=0$ where $\phi^*(z) = ((K_1-z)^+ - \psi^*(z))^+$. Then
$\sD \leq \Theta P_{\nu}(S^*) + (1-\Theta)P_\nu(R^*) = \sD(\psi^*)$.

Now consider $\sD(\psi^*)-A(u^*)$. Using \eqref{eqn:trivialRS2} for the second equality and the alternative characterisations of $\Theta$ for the third
we have
\begin{eqnarray*}
\sD(\psi^*) - A(u^*) & = & \Theta ( P_{\nu}(S^*) - P_{\nu}(R^*)) - (K_1-w)[P'_{\nu}(S^*) - P'_{\nu}(R^*)] - (K_2 -R^*)P'_{\nu}(R^*) \\
& = & P'_{\nu}(S^*)[\Theta(S^*-w) - (K_1-w)]  - P'_{\nu}(R^*)[\Theta(w-R^*) - (K_1-w) + (K_2-R^*)] \\
& = & 0.
\end{eqnarray*}
Then $\sD(\psi^*) = A(u^*) \leq \sup_{u \in [0,1]}A(u) \leq \sP_N \leq \sD \leq \sD(\psi^*)$. It follows that this chain of inequalities is in fact a chain of equalities and $\sP_{N} = \sup_{u \in [0,1]} A(u)$. Moreover, we have identified an optimal model and an optimal stopping rule. The model which yields the highest price for the American put is our extension of the left-curtain coupling.

\subsection{American puts with a general time-1 law }

We seek to generalise the arguments of the previous section to allow for non-trivial initial laws.
Define $\Lambda=\Lambda(r,g,s)$ via
\[ \Lambda(r,g,s) = \frac{K_1-g}{s-g} - \frac{(K_2 - r)-(K_1-g)}{g-r} \]
Suppose we are in the case of continuous $\mu$. Define $\hat{\Lambda}(x) = \Lambda(f(x),x,g(x))$ where $f$ and $g$ are the lower and upper functions which arise in the Beiglb\"{o}ck-Juillet~\cite{BeiglbockJuillet:16} characterisation of the left-curtain martingale coupling. In our notation this can be written as $\hat{\Lambda}(x) = \Lambda((R \circ G^{-1})(x),x,(S \circ G^{-1})(x))$. The fundamental insight in Hobson and Norgilas~\cite{HobsonNorgilas:17} is that, in the case of continuous $\mu$, the cheapest superchedge can be described in terms of a simple portfolio of European-style puts whose strikes depend on quantities which arise from looking for the root $x^*$, if any, of $\hat{\Lambda}( \cdot)=0$. Moreover the most expensive model is the model described by the left-curtain coupling, and an optimal exercise rule is to exercise at time-1 if and only if $X_1<x^*$. Hobson and Norgilas~\cite{HobsonNorgilas:17} identify four archetypes of hedging portfolios. The first two cases correspond to when there is a root to $\hat{\Lambda}=0$ and when $\hat{\Lambda}<0$ for all $x$. (The remaining cases correspond to cases where $\hat{\Lambda}$ is discontinuous, and jumps downwards over the value 0.)

In the case with atoms in $\mu$ we cannot use $\hat{\Lambda}$ directly since $G^{-1}$ has jumps. Instead, following the analysis in Section~\ref{ssec:trivial} we define $\bar{\Lambda}(u) = \Lambda(R(u),u,S(u))$, and look for solutions, if any, to $\bar{\Lambda}(\cdot)=0$. We may still have the cases where $\bar{\Lambda}<0$ for all $u \in (0,1)$ or where $\bar{\Lambda}(\cdot)$ jumps over zero, but these cases can be dealt with as in \cite{HobsonNorgilas:17}.
The new case is when the root $u^*$ of $\bar{\Lambda} =0$ occurs in an interval $(\underline{u}, \overline{u}]$ over which $G$ is constant. This means that there is an atom of $\mu$ at $G(u^*)$. See Figure~\ref{fig:strategy}. A model which maximises the price of the American put is the extended left-curtain martingale coupling model, and the optimal stopping rule is to exercise at time-1 whenever $X_1 < G(u)$ and to sometimes exercise when $X_1 = G(u)$. When $X_1=G(u)$ the optimal stopping rule is to exercise precisely when $U \in (\underline{u},u^*]$ and to wait if $U \in (u^*,\overline{u}]$. Because $R$ and $S$ are monotonic over $(\underline{u},\overline{u}]$ paths with low future variability are exercised at time-1 whereas on paths with high future variability exercise is delayed to time-2.

%In the previous section we showed how the extended left-curtain martingale coupling, and functions $R$, $G$ and $S$ can be used to find an optimal model, exercise strategy and a superhedge, under the assumption that $\mu=\delta_w$ and $\nu$ is an arbitrary probability measures. Combining this with similar methods used in Hobson and Norgilas \cite{HobsonNorgilas:17} (for continuous $\mu$) we have a complete guide how to find the optimal quantities for general $\mu$ and $\nu$ that leads to the highest robust price of the American put (i.e. solution to \eqref{eq:primal}).

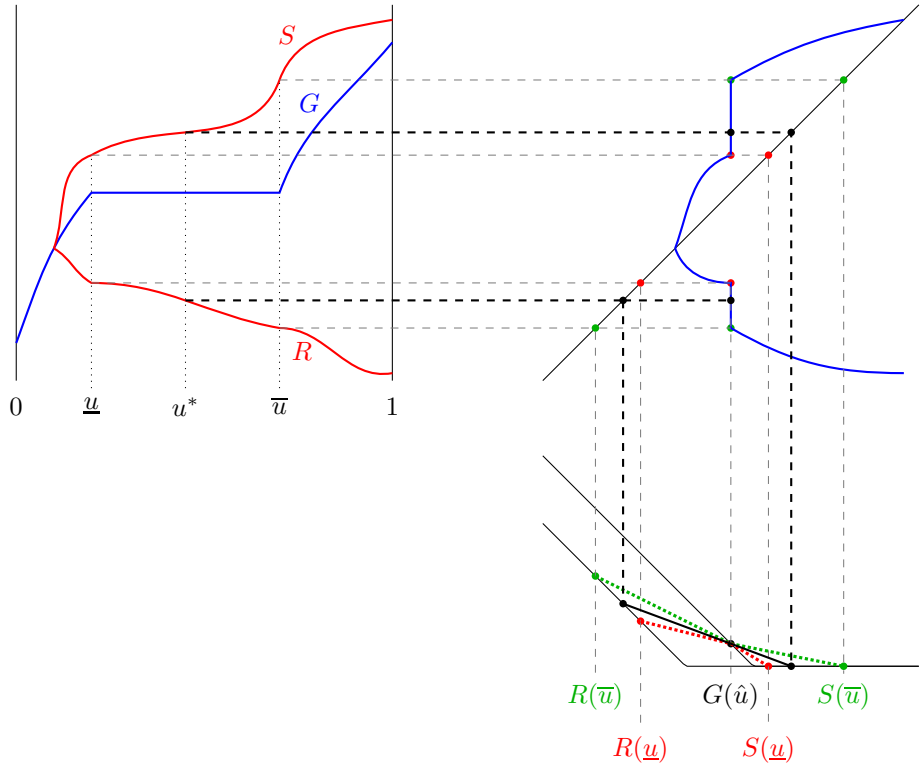
\begin{figure}[H]

\centering
\begin{tikzpicture}[scale=1,
declare function={	
k1=2.8;
k2=1.9;
Xop=-4.75;
a(\x)=((k1-\x)*(\x<k1))-k1-1;
b(\x)=((k2-\x)*(\x<k2))-k1-1;
x1=-6.5;
x2=-6.3;
z1=-5.2;
z2=-4.7;
 }]

               \draw[ black] (-7,0)--(-7,5);
               \draw[black] (-2,0)--(-2,5);

               \draw[name path=diag, black] (0,0) -- (5,5);
                %payoffs

   %a
	\draw[name path=a, black] plot[domain=0:5, samples=100] (\x,{a(\x)}) ;
	%b
	\draw[name path=b, black] plot[domain=0:5, samples=100] (\x,{b(\x)}) ;

    %G
    \draw[blue,thick, name path=g1] (-7,0.5) to[out=70, in=230] (-6,2.5) -- (-3.5,2.5) to[out=70, in=230] (-2,4.5)  ;

     %S
      \path [name path=lineA](x1,0) -- (x1,5);
	\path [gray, very thin, name intersections={of=lineA and g1}] (x1,0) -- (intersection-1);
	\coordinate (temp1) at (intersection-1);
	
	\draw[red,thick, name path=s0] (temp1) to[out=70, in=200] (-6,3) to[out=30, in=250] (-3.5,4) to[out=70, in=190] (-2,4.8) ;
	
	\path[red,thick, name path=s1] (temp1) to[out=70, in=200] (-6,3) ;
	\path[red,thick, name path=s2] (-6,3) to[out=0, in=250] (-3.5,4) ;
	\path[red,thick, name path=s3] (-3.5,4) to[out=70, in=190] (-2,4.8) ;
	
	%R
	
	\draw[red,thick, name path=r0] (temp1) to[out=330, in=150] (-6,1.3) to[out=0, in=170] (-3.5,0.7) to[out=0, in=190] (-2,0.1) ;

 %Circles

   \node (S1)[scale=0.3, shape=circle, fill, red] at (2.5,3) {} ;
    \node (dS1)[scale=0.3, shape=circle, fill, red] at (3,3) {} ;
     \node (bS1)[scale=0.3, shape=circle, fill, red] at (3,{b(3)}) {} ;
   \node (S3)[scale=0.3, shape=circle, fill, black!30!green] at (2.5,4) {} ;
   \node (dS3)[scale=0.3, shape=circle, fill, black!30!green] at (4,4) {} ;
     \node (bS3)[scale=0.3, shape=circle, fill, black!30!green] at (4,{b(4)}) {} ;

     \node (R1)[scale=0.3, shape=circle, fill, red] at (2.5,1.3) {} ;
    \node (dR1)[scale=0.3, shape=circle, fill, red] at (1.3,1.3) {} ;
     \node (bR1)[scale=0.3, shape=circle, fill, red] at (1.3,{b(1.3)}) {} ;
   \node (R3)[scale=0.3, shape=circle, fill, black!30!green] at (2.5,0.7) {} ;
   \node (dR3)[scale=0.3, shape=circle, fill, black!30!green] at (0.7,0.7) {} ;
     \node (bR3)[scale=0.3, shape=circle, fill, black!30!green] at (0.7,{b(0.7)}) {} ;
      \node (atm)[scale=0.3, shape=circle, fill, black] at (2.5,{a(2.5)}) {};

    % f and g
    \gettikzxy{(temp1)}{\l}{\k};
    \draw[blue,thick, name path=f] (\k,\k) to[out=70, in=200] (2.5,3)--(2.5,4) to[out=30, in=190] (4.8,4.8);
      \draw[blue,thick, name path=g] (\k,\k) to[out=290, in=180] (2.5,1.3)--(2.5,0.7) to[out=330, in=180] (4.8,0.1);

    %lines
 \draw[gray, dashed] (-6,1.3) -- (dR1) --  (1.3,-k1-1.8);
  \draw[gray, dashed] (dR1) --  (2.5,1.3);
  \draw[gray, dashed] (-3.5,0.7) -- (dR3) --  (0.7,-k1-1.1);
    \draw[gray, dashed](dR3) --  (2.5,0.7);
  \draw[gray, dashed] (-6,3) -- (dS1) --  (3,-k1-1.8);
  \draw[gray, dashed] (-3.5,4) -- (dS3) --  (4,-k1-1.1);
  \draw[gray, dashed] (R1) --  (2.5,-k1-1.1);

  \node[black!30!green,below] at (0.7,-k1-1.1) {$R(\overline{u})$};
   \node[black!30!green,below] at (4,-k1-1.1) {$S(\overline{u})$};
    \node[black, below] at (2.5,-k1-1.1) {$G(\hat{u})$};
\node[red,below] at (1.3,-k1-1.8) {$R(\underline{u})$};
\node[red,below] at (3,-k1-1.8) {$S(\underline{u})$};

  %candidate hedges
  \draw[black!30!green, very thick, densely dotted] (0.7,{b(0.7)}) -- (2.5,{a(2.5)}) -- (4,{b(4)});
    \draw[red, very thick, densely dotted] (1.3,{b(1.3)}) -- (2.5,{a(2.5)}) -- (3,{b(3)});

  %optimal hedge
   \path [name path=lineA](Xop,0) -- (Xop,5);
	\path [gray, very thin, name intersections={of=lineA and s0}] (Xop,0) -- (intersection-1);
	\coordinate (temp1) at (intersection-1);
	\gettikzxy{(temp1)}{\n}{\m};
	\draw[black,thick, dashed] (temp1) -- (\m,\m)-- (\m,{b(\m)});
	\node (opG)[scale=0.3, shape=circle, fill, black] at (2.5,\m) {};
	\node (opGd)[scale=0.3, shape=circle, fill, black] at (\m,\m) {};
	 \node (atmG)[scale=0.3, shape=circle, fill, black] at (\m,{b(\m)}){};
	
	 \path [name path=lineA](Xop,3) -- (Xop,-2);
	\path [gray, very thin, name intersections={of=lineA and r0}] (Xop,3) -- (intersection-1);
	\coordinate (temp3) at (intersection-1);
	\gettikzxy{(temp3)}{\z}{\u};
	\draw[black,thick, dashed] (temp3) -- (\u,\u);
	\draw[black,thick, dashed] (\u,\u)-- (2.5,\u);
	\node (opF)[scale=0.3, shape=circle, fill, black] at (2.5,\u) {};
	\node (opFd)[scale=0.3, shape=circle, fill, black] at (\u,\u) {};
	\path [name path=lineA](\u,\u) -- (\u,-3.8);
	\draw[black, thick, dashed, name intersections={of=lineA and b}] (\u,\u) -- (intersection-1);
	 \node (atmF)[scale=0.3, shape=circle, fill, black] at (intersection-1) {};
	\draw[black,thick] (\m,{b(\m)}) -- (2.5,{a(2.5)}) -- (intersection-1);

%Ticks
\draw[black,dotted] (-6,-0.1) -- (-6,3);
\draw[black, dotted] (Xop,-0.1) -- (Xop,\m);
\draw[black,dotted] (-3.5,-0.1) -- (-3.5,4);
\node[black,below] at (-6,-0.1) {$\underline{u}$};
\node[black,below] at (Xop,-0.1) {$u^*$};
\node[black,below] at (-3.5,-0.1) {$\overline{u}$};
\node[black,below] at (-7,-0.1) {$0$};
\node[black,below] at (-2,-0.1) {$1$};

\node[red] at (-3.4,4.6) {$S$};
\node[blue] at (-3.1,3.7) {$G$};
\node[red] at (-3.2,0.4) {$R$};

\end{tikzpicture}
\caption{Finding the optimal hedge for general measures. $\mu$ has an atom of size $\overline{u}-\underline{u}$. Moreover, the piecewise linear curve joining $(R(\underline{u}),K_2-R(\underline{u}))$, $(G(\hat{u}),K_1-G(\hat{u}))$ and $(S(\underline{u}),0)$ is concave (where $\hat{u}$ is any element of $(\underline{u},\overline{u}]$), whereas the piecewise linear curve joining $(R(\overline{u}),K_2-R(\overline{u}))$, $(G(\hat{u}),K_1-G(\hat{u}))$ and $(S(\overline{u}),0)$ is convex. There exists $u^*\in(\underline{u},\overline{u}]$ such that $(R(u^*),K_2-R(u^*))$, $(G(\hat{u}),K_1-G(\hat{u}))$ and $(S(u^*),0)$ all lie on a straight line. The figure describes the optimal coupling (via $(U,V)$ and \eqref{eq:YUVdef}) and the optimal exercise strategy for the American put is to exercise at time-1 if $U\leq u^*$.}
\label{fig:strategy}
\end{figure}

\appendix

\section{Proofs}
\label{app:proofs}
\begin{proof}[Proof of Lemma~\ref{lem:boundS}]
We begin our study of the upper bound on $S_n$ by considering the case of a single starting measure $\mu$ and fixed target law $\nu$. First we assume that $\mu$ and $\nu$ are regular (no atoms and no intervals within the support with no mass), before extending to the general case. Then we consider what happens when we consider $\mu_n \uparrow_{cx} \mu$.

%We prove that $S(u) \leq J(u)$ (where $J$ is constucted below in the case where $\mu$ and $\nu$ have no atoms, and there are no intervals between the upper and lower limits of the support which have no mass. {\bf This isn't the case of interest. We want to do atomic $\mu$, and then have an argument at the end about $\mu_n \uparrow \mu$}

%The general case requires a little more care, but the pictorial representation in Figure~\ref{fig:??} gives clear indications about how the argument may be extended.

Suppose $\mu$ and $\nu$ have no atoms and no intervals within the support with no mass. Then $G_\mu$ is continuous and strictly increasing.
Fix $u \in (0,1)$ and let $\ell_1 \equiv \ell^u_1$ be the tangent to $P_\mu$ with slope $u$. See Figure~\ref{fig:Sbound}.
By construction this tangent meets $P_\mu$ at $G= G_{\mu}(u)$. Let $H= H(u)$ be the point where the tangent crosses the $x$-axis.
Let $\ell_2 \equiv \ell^u_2$ be the tangent to $P_\nu$ with slope greater than $u$ which passes through $(G, P_{\mu}(G))$; this tangent meets $P_\nu$ at the $x$-coordinate $J=J(u)=J_{\mu,\nu}(u)$.

We now show that $S(u) \leq J$.

\begin{figure}[H]

\centering
%\resizebox{13 cm}{11cm}{
\begin{tikzpicture}

\begin{axis}[%
width=6.028in,
height=4.754in,
at={(1.011in,0.642in)},
scale only axis,
xmin=-38,
xmax=25,
ymin=-5,
ymax=20,
axis line style={draw=none},
ticks=none
]
\addplot [color=black, line width=1.0pt, forget plot]
  table[row sep=crcr]{%
-30	-0\\
-29.0823692733876	-0\\
-28	-0\\
-26.9530304567765	-0\\
-26	-0\\
-25.0188349687805	-0\\
-24	-0\\
-22.8966059152219	-0\\
-22	-0\\
-20.885479318973	-0\\
-20	-0\\
-19.035272075625	-0\\
-18	-0\\
-17.0386875053135	-0\\
-16	-0\\
-14.8844881899485	-0\\
-14	-0\\
-12.927495427025	-0\\
-12	-0\\
-10.9474394441761	-0\\
-10	-0\\
-9.00927578659704	-0\\
-8	-0\\
-6.99973779744308	-0\\
-6	-0\\
-4.96025116014284	-0\\
-4	-0\\
-2.91985035657147	-0\\
-2	-0\\
-1.01716014747721	-0\\
-0.508580073738607	-0\\
-0.254290036869304	-0\\
-0.127145018434652	-0\\
-0.0635725092173259	-0\\
-0.0317862546086629	-0\\
-0.0158931273043315	-0\\
-0.00794656365216574	-0\\
-0.00397328182608287	-0\\
-0.00198664091304143	-0\\
0	0\\
0.00196594828269691	0.00196594828269691\\
0.00393189656539382	0.00393189656539382\\
0.00786379313078763	0.00786379313078763\\
0.0157275862615753	0.0157275862615753\\
0.0314551725231505	0.0314551725231505\\
0.0629103450463011	0.0629103450463011\\
0.125820690092602	0.125820690092602\\
0.251641380185204	0.251641380185204\\
0.503282760370408	0.503282760370408\\
1.00656552074082	1.00656552074082\\
2	2\\
3.0007727756546	3.0007727756546\\
4	4\\
5.07294183468093	5.07294183468093\\
6	6\\
6.95283367288173	6.95283367288173\\
8	8\\
9.04935360608996	9.04935360608996\\
10	10\\
10.7548107460551	10.7548107460551\\
12	12\\
12.9814230456544	12.9814230456544\\
14	14\\
};
\addplot [color=black, line width=1.0pt, forget plot]
  table[row sep=crcr]{%
-30	0.00382154317047698\\
-29.0823692733876	0.00526505611314772\\
-28	0.00761086657781988\\
-26.9530304567765	0.0107650520115946\\
-26	0.0146388037201816\\
-25.0188349687805	0.0199247233647669\\
-24	0.0272044407581622\\
-22.8966059152219	0.0377399183401709\\
-22	0.0488700831654304\\
-20.885479318973	0.0667580480964812\\
-20	0.0849070261683935\\
-19.035272075625	0.1095346911613\\
-18	0.142755838976277\\
-17.0386875053135	0.181160803154195\\
-16	0.232419679601628\\
-14.8844881899485	0.300871866145184\\
-14	0.366681427084654\\
-12.927495427025	0.462343283809858\\
-12	0.56102450717163\\
-10.9474394441761	0.693355933545738\\
-10	0.833154705876863\\
-9.00927578659704	1.00260521280182\\
-8	1.20207233894765\\
-6.99973779744308	1.42885721266778\\
-6	1.68672732241755\\
-4.96025116014284	1.9902574142651\\
-4	2.30438836947453\\
-2.91985035657147	2.6983592831337\\
-2	3.06894635863276\\
-1.01716014747721	3.50146253407814\\
0	3.98942280401433\\
1.00656552074082	4.51289841844358\\
2	5.06894635863279\\
3.0007727756546	5.66808993946372\\
4	6.3043883694746\\
5.07294183468093	7.02849565875718\\
6	7.6867273224177\\
6.95283367288173	8.39307474893129\\
8	9.20207233894781\\
9.04935360608996	10.0446133042759\\
10	10.8331547058769\\
10.7548107460551	11.4749304393509\\
12	12.5610245071715\\
12.9814230456544	13.4385038840761\\
14	14.3666814270856\\
};
\addplot [color=black, line width=1.0pt, forget plot]
  table[row sep=crcr]{%
-30	0.292835548084406\\
-29.0823692733876	0.333269567149901\\
-28	0.38708680953244\\
-26.9530304567765	0.446090360946247\\
-26	0.50632471165126\\
-25.0188349687805	0.575428958367448\\
-24	0.655465115457608\\
-22.8966059152219	0.75249590557522\\
-22	0.839918096214468\\
-20.885479318973	0.960172818975746\\
-20	1.06552142738833\\
-19.035272075625	1.19087189006987\\
-18	1.33844423225056\\
-17.0386875053135	1.48833638002843\\
-16	1.66506137811206\\
-14.8844881899485	1.8729438649162\\
-14	2.05180134823564\\
-12.927495427025	2.28616367293635\\
-12	2.50497257645734\\
-10.9474394441761	2.77219530661013\\
-10	3.03057536342838\\
-9.00927578659704	3.31948441773478\\
-8	3.63410828667839\\
-6.99973779744308	3.96672285082317\\
-6	4.32037926204283\\
-4.96025116014284	4.71118893636708\\
-4	5.09333193373226\\
-2.91985035657147	5.54790993046339\\
-2	5.95589775986925\\
-1.01716014747721	6.41321460416528\\
0	6.90988298942701\\
1.00656552074082	7.42483065838808\\
2	7.95589775986943\\
3.0007727756546	8.51371241564439\\
4	9.09333193373246\\
5.07294183468093	9.74062669651958\\
6	10.3203792620426\\
6.95283367288173	10.9356707423562\\
8	11.6341082866767\\
9.04935360608996	12.3567715442201\\
10	13.0305753634232\\
10.7548107460551	13.5781485589542\\
12	14.5049725764452\\
12.9814230456544	15.2553315339212\\
14	16.0518013481812\\
};
\addplot [color=black!30!green, line width=1.0pt, dashed, forget plot]
  table[row sep=crcr]{%
-9.29415848085642	4.44089209850063e-16\\
-8.8083531269334	0.204397885765609\\
-8.23533309536295	0.445490500345778\\
-7.68105413026117	0.678697993174494\\
-7.17650770986947	0.890981000691556\\
-6.65706648866265	1.10953085103142\\
-6.117682324376	1.33647150103733\\
-5.5335314907928	1.58224729249052\\
-5.05885693888253	1.78196200138311\\
-4.46881554401808	2.03021618930134\\
-4.00003155338905	2.22745250172889\\
-3.48929234517774	2.44234106459256\\
-2.94120616789558	2.67294300207467\\
-2.43227513151251	2.88707079419794\\
-1.88238078240211	3.11843350242044\\
-1.29181467125193	3.36690845962119\\
-0.823555396908631	3.56392400276622\\
-0.255757862946741	3.80281930218511\\
0.235269988584843	4.009414503112\\
0.792508906722596	4.24386736744112\\
1.29409537407832	4.45490500345778\\
1.81859734766537	4.67558411622456\\
2.35292075957179	4.90039550380355\\
2.88247226568023	5.12319915835058\\
3.41174614506526	5.34588600414933\\
3.96220237815434	5.57748511960029\\
4.47057153055874	5.79137650449511\\
5.04241646185566	6.03197470704475\\
5.52939691605221	6.23686700484089\\
6.0497248089151	6.45578991367096\\
6.58822230154568	6.68235750518667\\
7.1211108643071	6.90656519391948\\
7.64704768703916	7.12784800553244\\
8.17686949702609	7.35076538781184\\
8.70587307253263	7.57333850587822\\
9.27390209839169	7.81233120326518\\
9.76469845802611	8.018829006224\\
10.2691406985262	8.23106818106319\\
10.8235238435196	8.46431950656978\\
11.3790649617632	8.69805803807811\\
11.8823492290131	8.90981000691555\\
12.2819556185963	9.07794051537878\\
12.9411746145065	9.35530050726133\\
13.4607524318301	9.57390782909105\\
14	9.80079100760711\\
};
\addplot [color=blue, line width=1.0pt, dashed, forget plot]
  table[row sep=crcr]{%
-5.68380747191276	0\\
-5.23158638086333	0.368369276102761\\
-4.69817985955309	0.802870404008083\\
-4.18221881400275	1.22316083408409\\
-3.71255224719342	1.60574080801617\\
-3.2290205736676	1.99961499052305\\
-2.72692463483375	2.40861121202425\\
-2.18315679619793	2.85155243933722\\
-1.74129702247407	3.21148161603234\\
-1.192045843591	3.65888945075811\\
-0.755669410114404	4.01435202004042\\
-0.280238169775211	4.40162776924082\\
0.229958202245268	4.81722242404851\\
0.703706271679959	5.20312709954199\\
1.21558581460494	5.62009282805659\\
1.76532543555499	6.06789853686252\\
2.20121342696461	6.42296323206467\\
2.7297584877177	6.85350432196716\\
3.18684103932428	7.22583363607276\\
3.70555741307464	7.64836849542143\\
4.17246865168395	8.02870404008084\\
4.66071122216559	8.42641561481855\\
5.15809626404362	8.83157444408893\\
5.65103928726352	9.23311490342937\\
6.14372387640329	9.63444484809701\\
6.65612645964437	10.0518366336585\\
7.12935148876296	10.4373152521051\\
7.66166414578477	10.8709253424094\\
8.11497910112263	11.2401856561132\\
8.59933614970962	11.6347321708483\\
9.10060671348231	12.0430560601213\\
9.59665609892797	12.4471268932702\\
10.086234325842	12.8459264641293\\
10.5794289665335	13.2476718854844\\
11.0718619382016	13.6487968681374\\
11.6006224875603	14.0795134902812\\
12.0574895505613	14.4516672721455\\
12.5270591395505	14.834168250095\\
13.043117162921	15.2545376761536\\
13.5602531075667	15.675785152988\\
14.0287447752807	16.0574080801617\\
14.4007259319895	16.3604156844791\\
15.0143723876403	16.8602784841698\\
15.4980312142419	17.2542562427535\\
16	17.6631488881779\\
};
\addplot [color=blue, line width=1.0pt, dashed, forget plot]
  table[row sep=crcr]{%
-24.3867716662262	-0.33403600634423\\
-24.0345939263577	-0.310507997194423\\
-23.6191911359432	-0.282756099367565\\
-23.2173744173544	-0.255911848975614\\
-22.8516106056602	-0.2314761923909\\
-22.4750490181809	-0.206319166626055\\
-22.0840300753772	-0.180196285414236\\
-21.6605581670246	-0.151905312401224\\
-21.3164495450942	-0.128916378437571\\
-20.8887073574171	-0.100340120014254\\
-20.5488690148112	-0.0776364714609067\\
-20.1786158289259	-0.0529008923510377\\
-19.7812884845282	-0.0263565644842422\\
-19.4123461073086	-0.00170855683272753\\
-19.0137079542452	0.0249233424924223\\
-18.585585380897	0.053525013417829\\
-18.2461274239621	0.076203249469087\\
-17.8345106095346	0.10370221683619\\
-17.4785468936791	0.127483156445752\\
-17.074584398882	0.15447076014073\\
-16.7109663633961	0.178763063422416\\
-16.3307360548521	0.204165186173833\\
-15.9433858331131	0.230042970399081\\
-15.5594949371828	0.255689646748777\\
-15.1758053028301	0.281322877375745\\
-14.7767598199008	0.30798198926923\\
-14.4082247725471	0.33260278435241\\
-13.9936738545032	0.360297770970356\\
-13.6406442422641	0.383882691329074\\
-13.2634398746727	0.409082659434238\\
-12.8730637119811	0.435162598305739\\
-12.4867536638937	0.460970891440492\\
-12.1054831816981	0.486442505282403\\
-11.7213963327832	0.512102272702577\\
-11.3379026514151	0.537722412259068\\
-10.9261180202014	0.565232590995973\\
-10.570322121132	0.589002319235732\\
-10.204633833181	0.613432930290537\\
-9.80274159084903	0.640282226212397\\
-9.40000989214057	0.667187603865357\\
-9.03516106056602	0.691562133189062\\
-8.74547204420588	0.710915445610408\\
-8.26758053028301	0.742842040165726\\
-7.89091991937534	0.768005681408682\\
-7.5	0.794121947142391\\
};
\addplot [color=blue, line width=1.0pt,dashed, forget plot]
  table[row sep=crcr]{%
-29.5763548608086	-0.966190576930408\\
-29.0325261231193	-0.861102136566547\\
-28.3910660034991	-0.737147591955802\\
-27.7705853367365	-0.617247077277096\\
-27.2057771461896	-0.508104606981196\\
-26.6242951568466	-0.395740123229593\\
-26.0204882888801	-0.279061622006589\\
-25.3665679319258	-0.152699284616137\\
-24.8351994315707	-0.0500186370319828\\
-24.1746849593395	0.0776179347671981\\
-23.6499105742612	0.179024347942623\\
-23.0781699447127	0.289506429686221\\
-22.4646217169517	0.40806733291723\\
-21.8949052227795	0.518158274555422\\
-21.2793328596422	0.637110317891836\\
-20.6182310003166	0.764860395266151\\
-20.0940440023327	0.866153302866443\\
-19.4584301424524	0.988978127262999\\
-18.9087551450233	1.09519628784105\\
-18.2849609957925	1.21573709362727\\
-17.7234662877138	1.32423927281566\\
-17.1363191023072	1.43769848837787\\
-16.5381774304043	1.55328225779026\\
-15.945377608865	1.66783377810572\\
-15.3528885730948	1.78232524276487\\
-14.7366872159532	1.90139883171725\\
-14.1675997157853	2.01136822773948\\
-13.527455047494	2.13506857701454\\
-12.9823108584758	2.24041121271408\\
-12.3998362956184	2.35296749950099\\
-11.7970220011664	2.46945419768869\\
-11.2004865532234	2.58472758341019\\
-10.6117331438569	2.69849718266329\\
-10.0186307340158	2.81310717457192\\
-9.4264442865474	2.9275401676379\\
-8.79057128595315	3.05041506789763\\
-8.24115542923792	3.15658315261251\\
-7.67646386156993	3.26570308692308\\
-7.05586657192844	3.38562613758711\\
-6.43397300359047	3.50579967870347\\
-5.87057771461896	3.61466912256172\\
-5.42324333128068	3.7011111757454\\
-4.68528885730948	3.84371210753633\\
-4.10365395714906	3.9561061394861\\
-3.5	4.07275509251093\\
};
\addplot [color=red, line width=1.0pt, dashed, forget plot]
  table[row sep=crcr]{%
-7.5	3.32102726161298\\
-7.00990177101383	3.55479423958886\\
-6.43181818181818	3.83052846116169\\
-5.87264126668744	4.09724458024335\\
-5.36363636363636	4.34002966071039\\
-4.83960504014415	4.58998204089117\\
-4.29545454545454	4.8495308602591\\
-3.70614179562987	5.1306211651438\\
-3.22727272727273	5.3590320598078\\
-2.63201736354238	5.64295687176036\\
-2.15909090909091	5.86853325935651\\
-1.6438384949361	6.11429827671009\\
-1.09090909090909	6.37803445890521\\
-0.577480826701524	6.62292939349718\\
-0.0227272727272725	6.88753565845392\\
0.573057444004799	7.17171296111992\\
1.04545454545454	7.39703685800263\\
1.61826948783893	7.67025804122875\\
2.11363636363636	7.90653805755133\\
2.6757993877696	8.17467849044631\\
3.18181818181818	8.41603925710004\\
3.71095497761295	8.66842684467542\\
4.25	8.92554045664874\\
4.7842309490929	9.18035785268173\\
5.31818181818182	9.43504165619745\\
5.87350222128735	9.69991829676575\\
6.38636363636364	9.94454285574616\\
6.9632617413766	10.2197116252556\\
7.45454545454546	10.4540440552949\\
7.97947128487013	10.7044230972072\\
8.52272727272727	10.9635452548436\\
9.0603247667593	11.2199684249645\\
9.59090909090909	11.4730464543923\\
10.1254127324519	11.7279939192281\\
10.6590909090909	11.982547653941\\
11.2321393889773	12.2558802298489\\
11.7272727272727	12.4920488534897\\
12.2361725298346	12.7347838031415\\
12.7954545454545	13.0015500530384\\
13.355904766889	13.2688735135652\\
13.8636363636364	13.5110512525871\\
14.2667739211885	13.7033397428608\\
14.9318181818182	14.0205524521358\\
15.4559873084745	14.2705705616486\\
16	14.5300536516845\\
};
\node[right, align=left, font=\color{blue}]
at (axis cs:-24.5,-1) {$\ell^\gamma_3$};
\node[right, align=left]
at (axis cs:-16.5,-1.5) {$r(\gamma)$};
\node[right, align=left]
at (axis cs:-3.5,-4) {$G(u)$};
\node[right, align=left, font=\color{red}]
at (axis cs:11,-1.5) {$J(u)$};
\node[right, align=left]
at (axis cs:-8.4,-1.5) {$\gamma$};
\node[right, align=left, font=\color{red}]
at (axis cs:-1.8,-1.5) {$s$};
\node[right, align=left, font=\color{blue}]
at (axis cs:16,18.052) {$\ell_2^u$};
\node[right, align=left]
at (axis cs:13,14.867) {$P_\mu$};
\node[right, align=left, font=\color{red}]
at (axis cs:16,14.867) {$\ell_5$};
\node[right, align=left, font=\color{blue}]
at (axis cs:-31,-1.5) {$\ell_4^\gamma$};
\node[right, align=left, font=\color{black!30!green}]
at (axis cs:14,9.828) {$\ell^u_1$};
\node[right, align=left]
at (axis cs:-31,0.8) {$P_\nu$};
\node[right, align=left]
at (axis cs:-10.494,-1.5) {$H$};
\addplot [color=black, dotted, forget plot]
  table[row sep=crcr]{%
-7.5	-1\\
-7.5	3.3\\
};
\addplot [color=black, dotted, forget plot]
  table[row sep=crcr]{%
-15	-1\\
-15	1.893067937626047\\
};
\addplot [color=black, dotted, forget plot]
  table[row sep=crcr]{%
-2	-3.4\\
-2	3.06894635863276\\
};
\addplot [color=black, dotted, forget plot]
  table[row sep=crcr]{%
13	-1\\
13	18\\
};
\addplot [color=black, dotted, forget plot]
  table[row sep=crcr]{%
-1	-1\\
-1	6.42139626312213\\
};
\addplot [color=black, dotted, forget plot]
  table[row sep=crcr]{%
-9.29415848085642	0\\
-9.29415848085642	-1\\
};
\end{axis}
\end{tikzpicture}%
%}

\caption{Construction of function $J$ that bounds the upper function $S$ on $(0,1)$.}
\label{fig:Sbound}
\end{figure}

Choose $\gamma \in [H,G)$. Let $\ell_3^\gamma$ be the tangent to $P_\mu$ which passes through $(\gamma, \ell_1(\gamma))$ and has slope less than $u$.
Suppose this tangent meets $P_\mu$ at $r=r(\gamma)$; the slope of the tangent is $P_{\mu}'(r)$. Let $\ell_4^{\gamma}$ be the tangent to
$P_\nu$ at $r$. Finally, let $\ell_5^\gamma$ be the line passing through $(\gamma, \ell_4^{r}(\gamma))$ with slope $u+P_\nu'(r) - P_\mu'(r)$.

If there exists $\gamma$ such that $\ell_5^\gamma$ is a tangent to $P_\nu$ (meeting $P_\nu$ at $s$ say), then $(r,G,s)$ satisfy
\begin{equation} \int_r^G w^i \mu(dw) = \int_r^s w^i \nu(dw) \hspace{10mm} i=0,1  \label{eq:gammameanmass} \end{equation}
(and moreover $\gamma =  \int_r^G w \mu(dw)/ \int_r^G  \mu(dw) = \int_r^s w \nu(dw)/\int_r^s  \nu(dw)$ is the barycentre of the measures $\mu |_{(r,G)}$ and $\nu |_{(r,s)}$).

For each $u$ there may be multiple $\gamma$ which lead to a triple $(r,G,s)$ which satisfies \eqref{eq:gammameanmass}. We show that in each case $s \leq J$. It follows that $S(u) \leq J$.

Suppose $\ell_4^\gamma(\gamma) \leq \ell_3^\gamma(\gamma) = \ell_1(\gamma)$. Then necessarily $P_\nu'(r) < P_{\mu}'(r)$ and $\ell_5^\gamma$ lies below $\ell_1$ to the right of $\gamma$; in particular $\ell^\gamma_5$ stays below $P_\mu$ to the right of $\gamma$ and cannot be a tangent to $P_\nu$. Hence if $(r,G,s)$ satisfies \eqref{eq:gammameanmass} we must have $\ell_4^\gamma(\gamma) > \ell_1(\gamma)$. Then, if $\ell^\gamma_5$ is a tangent to $P_\nu$ we must have that the point of tangency is below $J$.

In the above we used the regularity assumptions on $\mu$ and $\nu$ to conclude that there was a unique tangent to $P_\cdot \in \{P_\mu , P_\nu \}$ at a given point, and that there was a unique point at which $P_{\cdot}$ had a given slope. If $\mu$ or $\nu$ is not regular then, for fixed u, there may be multiple quintiles $G$, multiple points $r$ and multiple tangents to $P_\nu$ at $r$. The point is that although there are multiple versions of the construction in this case each candidate triple $(r,G,s)$ satisfying \eqref{eq:gammameanmass} has $s \leq J$ where $J$ is defined using an arbitrary point $G \in [G_\mu(u), G_\mu(u+)]$. We define $J_- = J_-(u)$ to be the smallest $x$-coordinate at which the tangent to $P_\nu$ with slope greater than $u$ passing through $(G(u),P_\mu(G(u)))$ meets $P_\nu$ and $J_+ = J_+(u)$ to be the largest $x$-coordinate at which the tangent to $P_\nu$ with slope greater than $u$ passing through $(G(u+),P_\mu(G(u)+))$ meets $P_\nu$. We have $S(u) \leq J_-(u) \leq J_+(u)$.

Finally, we want to show that if we approximate $\mu$ by $\mu_n$ (with $\mu_n \uparrow_{cx} \mu$) then the bound $\limsup S_n \leq J_+$ remains valid, where $J_+$ is constructed from $\mu$ and $\nu$.

Define $K(k) = \overline{\argsup}_\kappa \frac{P_\nu(\kappa) - P_\mu(k)}{\kappa - k}$. The notation $\overline{\argsup}$ is used to indicate that where there are multiple elements in the $\argsup$ we choose the largest one. Then $K$ is increasing and right continuous in $k$. Note that $J_+(u) = K(G(u+))$.
In a similar fashion we can define $K_n$ and $J_n$ using $P_{\mu_n}$ in place of $P_\mu$. (The target law is assumed fixed throughout.) Since $P_{\mu_n}(k) \uparrow P_\mu(k)$ and $K$ is right-continuous we have $K_n(k) \downarrow K(k)$. Then, for $\epsilon>0$,
\[ \limsup_n J_n(u) = \limsup_n K_n(G_n(u+)) \leq \limsup_n K_n (G(u+)+\epsilon) \leq K(G(u+)+\epsilon) .\]
Since $\epsilon$ is arbitrary and $K$ is right continuous, $\limsup S_n(u) \leq \limsup_n J_n(u) \leq J_+(u)$.

\end{proof}

\begin{proof}[Proof of Lemma~\ref{lem:boundR}]
As for the proof of Lemma~\ref{lem:boundS} we begin by considering a single initial law $\mu$, and supposing that $\mu$ and $\nu$ are regular.

Fix $u \in (0,1)$ and let $\ell_1$ be the tangent to $P_\mu$ with slope $u$. Let $H=H(u)$ be the point where this tangent crosses the $x$-axis. Suppose that $\ell_1$ is not a tangent to $P_\nu$. Then $\ell_1$ must lie strictly below $P_\nu$. There exists $\epsilon=\epsilon(u)>0$ such that the line passing through $(H,\epsilon)$ with slope $u+\epsilon$ lies below $P_\nu$. Now choose $j=j(u)$ such that the tangents to $P_{\mu}$ and $P_\nu$ at $j$ both have slope less than $\epsilon$ and both cross the line $y=x$ below $\epsilon$. Then $R(u) \geq j$.

To see this let $\gamma$ be the $x$-coordinate of the point where the tangent to $P_\mu$ at $j$ crosses $\ell_1$. Then if $\ell_4$ is the tangent to $P_\nu$ at $j$ then $\ell_4(\gamma)<\epsilon$; if $\ell_5$ is the line passing through $(\gamma, \ell_4(\gamma))$ with slope $u + P_{\nu}'(j) - P_{\mu}'(j)<u+\epsilon$, then by our defining assumption on $\epsilon$, $\ell_5$ lies below $P_\nu$. Hence $R(u)>j$.

We can extend the result to irregular measures, and to $\liminf R_n(u)$ by similar techniques as for $S$. The only extra issue that arises is our assumption that $\ell^1$ is not a tangent to $P_\nu$. But, if for each $n$, $\ell^1$ is a tangent to $P_\nu$, then the same is certainly true in the limit. Then there must exist $x$ such that $\ell_1(x) = P_\mu(x)=P_\nu(x)$ and then $S(u) \leq x \leq G(u+)$. This case is excluded by hypothesis.
\end{proof}
\end{document}